\documentclass[11pt]{amsart}

\usepackage[utf8]{inputenc}
\usepackage[T1]{fontenc}
\usepackage{lmodern}
\usepackage[margin=0.75in]{geometry}

\usepackage{amsmath}
\usepackage{amssymb}
\usepackage{amsthm}
\usepackage{amsfonts}
\usepackage{bm}

\usepackage{graphicx}
\usepackage{subcaption}

\usepackage{booktabs}
\usepackage{array}
\usepackage{multirow}
\usepackage{longtable}

\usepackage{algorithm}
\usepackage{algpseudocode}

\usepackage[numbers,sort&compress]{natbib}

\usepackage[colorlinks=true, linkcolor=blue, citecolor=blue, urlcolor=blue]{hyperref}
\usepackage{doi}

\theoremstyle{plain}
\newtheorem{theorem}{Theorem}[section]
\newtheorem{lemma}[theorem]{Lemma}
\newtheorem{proposition}[theorem]{Proposition}

\theoremstyle{definition}
\newtheorem{definition}[theorem]{Definition}

\theoremstyle{remark}
\newtheorem{remark}[theorem]{Remark}

\newcommand{\bu}{\mathbf{u}}

\newcommand{\CE}{\mathcal{E}}

\newcommand{\butchertableau}[1]{%
  \begingroup
  \setlength{\arraycolsep}{3pt}%
  \renewcommand{\arraystretch}{1.15}%
  \small
  \[
  \begin{array}{@{}c|*{8}{c}@{}}
  #1
  \end{array}
  \]
  \endgroup
}

\title[A~Posteriori Error Analysis of RKDG Schemes with SIAC Post-Processing]{A~Posteriori Error Analysis of Runge--Kutta Discontinuous Galerkin Schemes with SIAC Post-Processing for Nonlinear Convection--Diffusion Systems}

\author{Jan Giesselmann}
\address{Department of Mathematics\\
Technische Universit\"at Darmstadt\\
Dolivostra\ss e 15\\
64293 Darmstadt, Germany}
\email{giesselmann@mathematik.tu-darmstadt.de}

\author{Kiwoong Kwon}
\address{Department of Mathematics\\
Technische Universit\"at Darmstadt\\
Dolivostra\ss e 15\\
64293 Darmstadt, Germany}
\email{kwon@mathematik.tu-darmstadt.de}

\author{Sebastian Krumscheid}
\address{Institute for Applied and Numerical Mathematics\\
Karlsruhe Institute of Technology (KIT)\\
76131 Karlsruhe, Germany}
\email{sebastian.krumscheid@kit.edu}


\begin{document}

\begin{abstract}
  We develop reliable a~posteriori error estimators for fully discrete Runge--Kutta discontinuous Galerkin approximations of nonlinear convection--diffusion systems endowed with a convex entropy in multiple spatial dimensions on the flat torus $\mathbb{T}^d$, with a focus on the convection-dominated regime.
  In order to use the relative entropy method, we reconstruct the numerical solution via tensor-product Smoothness-Increasing Accuracy-Conserving (SIAC) filtering which has superconvergence properties.
  We then derive reliable a~posteriori error estimators for the difference between the entropy weak solution and the reconstruction, with constants that are uniform in the vanishing viscosity limit.
  Our numerical experiments show that the a~posteriori error bounds converge with the same order as the error of the reconstructed numerical solution.
\end{abstract}

\maketitle

\section{Introduction}
\label{sec:introduction}
We develop reliable a~posteriori error estimators for fully discrete Runge--Kutta discontinuous Galerkin (RKdG) approximations of nonlinear convection--diffusion systems.
The general class of systems considered in this work is given by
\begin{equation}
\label{eq:intro-convection-diffusion}
\partial_t \mathbf{u} + \sum_{\alpha=1}^{d} \partial_{x_\alpha} \mathbf{f}_\alpha(\mathbf{u})
  = \varepsilon \sum_{\alpha, \beta=1}^{d} \partial_{x_\alpha}
  \bigl( \mathbf{A}_{\alpha\beta}(\mathbf{u}) \partial_{x_\beta} \mathbf{u} \bigr)
\end{equation}
where $\varepsilon$ denotes the diffusion coefficient, together with suitable initial data and periodic boundary conditions.
Our main interest is the convection-dominated regime, in which $\varepsilon$ is small.
In this regime, we seek a~posteriori error estimators whose constants do not blow up as $\varepsilon \to 0$.

There is extensive literature on a~posteriori error estimation for convection--diffusion problems, especially in scalar or linear settings \cite{2013Verfurth,2008Sangalli,2013Dolejsi,2014Cangiani}.
Traditional approaches typically rely on energy estimates.
However, for nonlinear systems in convection-dominated regimes, this approach does not lead to a~posteriori estimates that are robust in the vanishing-viscosity limit.
We therefore adopt a hyperbolic viewpoint through the relative entropy framework.
This framework can be traced back to the works of Dafermos and DiPerna \cite{1979Dafermos,1979DiPerna} and provides a weak--strong stability theory that is well suited to convection-dominated problems.

Since the discontinuous Galerkin approximation is discontinuous across element interfaces, it is not regular enough to be used directly in the relative entropy stability estimate. Instead, one needs to construct a sufficiently regular reconstruction.

It is important to note that such a reconstruction can be achieved in many ways but most will lead to error estimators that decay more slowly than the true error.

In one space dimension for hyperbolic problems, i.e. $\varepsilon=0$, error estimators that scale optimally can be achieved by a flux-based, moment-preserving piecewise polynomial reconstruction \cite{2015Giesselmann,2016Dedner}.
This also forms the basis of our recent a~posteriori analysis of one-dimensional nonlinear convection--diffusion problems \cite{2025Dedner}.

Multidimensional extensions of this reconstruction can be defined on Cartesian meshes \cite{2017Giesselmann}, but they do not in general yield a~posteriori error estimators that converge at the same rate as the discrete error.
This motivates the main contribution of this work: a novel reconstruction strategy for multidimensional problems that leads to a~posteriori estimators that scale optimally, i.e. with the same rate as the error, at least experimentally, for convection-diffusion problems as well as for systems of hyperbolic conservation laws, as long as the exact solution is Lipschitz continuous.

Our main idea is to use Smoothness-Increasing Accuracy-Conserving (SIAC) filtering for the spatial reconstruction \cite{2003Cockburn,2015Ryan}.
On Cartesian meshes, tensor-product SIAC kernels provide a natural multidimensional reconstruction operator \cite{2012Ji}, and their smoothing properties provide the regularity required by the relative entropy argument.
SIAC filtering has also been studied on structured meshes and on adaptive meshes constructed in a hierarchical manner~\cite{2012King}.
Combined with a temporal Hermite reconstruction \cite{2016Dedner}, this yields a space--time reconstruction suitable for the a~posteriori analysis.
Since SIAC filtering is linear, the resulting space--time reconstruction is a piecewise polynomial in space and time.
In the present work, following the viewpoint of \cite{2021Dedner}, we treat the space--time reconstruction as the primary numerical approximation and derive a~posteriori bounds for its error.

In order to obtain error estimators that not only scale with optimal order in the mesh width but also remain robust as $\varepsilon \to 0$, we split the residual of the space--time reconstruction into hyperbolic and parabolic contributions and treat them in different norms.
This splitting was already employed in \cite{2025Dedner}.
Without it, the estimator would be expected to contain terms of order $\varepsilon^{-1}$ or to exhibit a reduced convergence rate.
See Remark~\ref{rem:splitting} for details.
We perform the splitting by adapting the Thom\'ee-type derivative post-processing strategy of \cite{2009Ryan} to define an auxiliary reconstruction that yields a modified diffusive flux.

Defining a sufficiently smooth reconstruction and splitting the residual allows us to apply multidimensional analogues of the relative entropy stability estimates from \cite{2025Dedner} and thereby derive reliable a~posteriori bounds for specific systems of the form \eqref{eq:intro-convection-diffusion}.
Lemma 3.7 in \cite{2025Dedner} provides stability estimates for a wide variety of problems of the form \eqref{eq:intro-convection-diffusion}, and the periodic boundary conditions considered here remove complications due to boundary fluxes.
As can be seen in \cite{2025Dedner}, for many specific systems the estimate \cite[Lemma 3.7]{2025Dedner} can be simplified and the stability estimates can be optimized for the specific system.
While our reconstruction methodology is applicable to all systems of the form \eqref{eq:intro-convection-diffusion} that allow appropriate stability estimates, we provide specific estimates and numerical experiments for three instructive special cases: linear scalar advection (either inviscid or with linear diffusion), the viscous or inviscid Burgers equation, and a multidimensional $p$-system (inviscid or with linear viscosity).

While we are able to prove that our error estimator provides an upper bound for the error, we are not able to show that the error estimator provides a lower bound for the error, i.e., that it is efficient.
To assess the efficiency of the estimator, we conduct numerical experiments in two space dimensions.
In these experiments, we observe that the residuals, which are key components of the a~posteriori error estimator, converge at the same rate as the error of the reconstructed numerical solution in convection-dominated regimes over a range of viscosities.

The rest of the paper is organized as follows. In Section~\ref{sec:preliminaries}, we introduce the continuous problem and the entropy structure. In Section~\ref{sec:discretization}, we present the RKdG discretization.
Since our focus is on convection-diffusion problems, we assume that the polynomial degree $q$ of the discontinuous Galerkin approximation satisfies $q \geq 1$.
In Section~\ref{sec:reconstruction}, we develop the space--time reconstruction and the residual splitting.
In Section~\ref{sec:relative-entropy}, we state the stability estimates used in the analysis.
Finally, in Section~\ref{sec:numerical}, we present numerical experiments.

\section{Preliminaries}
\label{sec:preliminaries}

In this section, we introduce the continuous problem setting for nonlinear convection--diffusion systems with periodic boundary conditions and recall the notion of entropy weak solution used in the subsequent a~posteriori analysis.
Throughout the paper, we identify $\mathbb{T}^d$ with the flat torus, i.e., the unit cube $[0,1]^d$ endowed with periodic boundary conditions.
We consider nonlinear convection--diffusion systems of the following form:
\begin{equation}
\label{eq:convection-diffusion}
\partial_t \mathbf{u} + \sum_{\alpha=1}^{d} \partial_{x_\alpha} \mathbf{f}_\alpha(\mathbf{u})
  = \varepsilon \sum_{\alpha, \beta=1}^{d} \partial_{x_\alpha} \bigl( \mathbf{A}_{\alpha\beta}(\mathbf{u}) \partial_{x_\beta} \mathbf{u} \bigr)
  \quad \text{on } \mathbb{T}^d \times (0, T),
\end{equation}
where the unknown is $\mathbf{u} : \mathbb{T}^d \times (0,T) \to U \subset \mathbb{R}^m$ and the  open set $U$ denotes the set of admissible states.
For each $\alpha = 1,\ldots,d$, the mapping $\mathbf{f}_\alpha \in C^2(U;\mathbb{R}^m)$ denotes the convective flux in the $\alpha$-th spatial direction, $\varepsilon \ge 0$ is the diffusion coefficient, and $\mathbf{A}_{\alpha\beta} \in C^1(U;\mathbb{R}^{m\times m})$ for $\alpha,\beta = 1,\ldots,d$ are the components of the diffusion tensor.

We restrict attention to systems of the form \eqref{eq:convection-diffusion} that admit a strictly convex entropy/entropy-flux pair $(\eta,\mathbf{q})$, where $\eta \in C^2(U;\mathbb{R})$ and $\mathbf{q} \in C^1(U;\mathbb{R}^d)$ satisfy
\begin{equation}
\label{eq:entropy-compatibility}
D q_\alpha = D\eta \, D\mathbf{f}_\alpha,
\end{equation}
for $\alpha = 1,\ldots,d$.
We further assume that the entropy is compatible with the diffusion tensor $\{\mathbf{A}_{\alpha\beta}\}_{\alpha,\beta=1}^{d}$ in the sense that
\begin{equation}
\label{eq:positive-semidefinite}
\sum_{\alpha, \beta=1}^{d} \boldsymbol{\xi}_\alpha^\top \, D^2\eta(\mathbf{u}) \, \mathbf{A}_{\alpha\beta}(\mathbf{u}) \, \boldsymbol{\xi}_\beta \geq 0
\end{equation}
for all $\mathbf{u} \in U$ and all $\boldsymbol{\xi} = (\boldsymbol{\xi}_1,\ldots,\boldsymbol{\xi}_d) \in (\mathbb{R}^m)^d$.
Under this assumption, the diffusive contribution is entropy dissipative, and smooth solutions of \eqref{eq:convection-diffusion} satisfy the entropy balance
\begin{equation}
\label{eq:entropy-balance}
\partial_t \eta(\mathbf{u}) + \sum_{\alpha=1}^{d} \partial_{x_\alpha} \biggl( q_\alpha(\mathbf{u})
  - \varepsilon \sum_{\beta=1}^{d} D\eta(\mathbf{u}) \, \mathbf{A}_{\alpha\beta}(\mathbf{u}) \, \partial_{x_\beta} \mathbf{u} \biggr)
= -\varepsilon \sum_{\alpha, \beta=1}^{d} (\partial_{x_\alpha} \mathbf{u})^\top D^2\eta(\mathbf{u}) \, \mathbf{A}_{\alpha\beta}(\mathbf{u}) \, \partial_{x_\beta} \mathbf{u}.
\end{equation}
This motivates the following notion of entropy weak solution.

\begin{definition}[Weak solution and entropy weak solution]
\label{def:entropy-solution}
A function $\mathbf{u} \in L^\infty((0,T) \times \mathbb{T}^d; U)$ is called a \emph{weak solution}
of \eqref{eq:convection-diffusion} if, for each $\alpha = 1, \ldots, d$,
\begin{equation*}
\varepsilon \sum_{\beta=1}^{d} \mathbf{A}_{\alpha\beta}(\mathbf{u}) \, \partial_{x_\beta} \mathbf{u}
\in L^1((0,T) \times \mathbb{T}^d; \mathbb{R}^m)
\qquad \text{whenever } \varepsilon > 0,
\end{equation*}
and if, for all $\phi \in C_0^\infty([0,T) \times \mathbb{T}^d; \mathbb{R}^m)$, the following identity holds:
\begin{equation}
\label{eq:weak-form}
\int_0^T \int_{\mathbb{T}^d} \partial_t \phi \cdot \mathbf{u}
  + \sum_{\alpha=1}^{d} \partial_{x_\alpha} \phi \cdot \biggl( \mathbf{f}_\alpha(\mathbf{u})
  - \varepsilon \sum_{\beta=1}^{d} \mathbf{A}_{\alpha\beta}(\mathbf{u}) \, \partial_{x_\beta} \mathbf{u} \biggr) \, \mathrm{d}x \, \mathrm{d}t
  + \int_{\mathbb{T}^d} \mathbf{u}_0 \cdot \phi(0, \cdot) \, \mathrm{d}x = 0.
\end{equation}

A weak solution is called an \emph{entropy weak solution} with respect to an
entropy/entropy-flux pair $(\eta, \mathbf{q})$ if, in addition,
\begin{equation*}
\varepsilon \sum_{\alpha, \beta=1}^{d} \bigl( \partial_{x_\alpha}(D\eta(\mathbf{u})) \bigr)^\top \mathbf{A}_{\alpha\beta}(\mathbf{u}) \, \partial_{x_\beta} \mathbf{u} \in W^{-1,1}((0,T) \times \mathbb{T}^d),
\qquad \text{whenever } \varepsilon > 0,
\end{equation*}
and if, for all non-negative $\Phi \in C_0^\infty([0,T) \times \mathbb{T}^d; \mathbb{R}_+)$, the following entropy inequality holds:
\begin{align}
\label{eq:entropy-inequality}
&\int_0^T \int_{\mathbb{T}^d} \partial_t \Phi \, \eta(\mathbf{u})
  + \sum_{\alpha=1}^{d} \partial_{x_\alpha} \Phi \biggl( q_\alpha(\mathbf{u})
  - \varepsilon \sum_{\beta=1}^{d} D\eta(\mathbf{u}) \, \mathbf{A}_{\alpha\beta}(\mathbf{u}) \, \partial_{x_\beta} \mathbf{u} \biggr) \notag \\
&\quad - \varepsilon \int_0^T \int_{\mathbb{T}^d} \Phi \sum_{\alpha, \beta=1}^{d} \bigl( \partial_{x_\alpha}(D\eta(\mathbf{u})) \bigr)^\top \mathbf{A}_{\alpha\beta}(\mathbf{u}) \, \partial_{x_\beta} \mathbf{u} \, \mathrm{d}x \, \mathrm{d}t
  + \int_{\mathbb{T}^d} \Phi(0, \cdot) \eta(\mathbf{u}_0) \, \mathrm{d}x \geq 0.
\end{align}
\end{definition}

\begin{remark}[Well-posedness]
\label{rem:well-posedness}
The existence and uniqueness of entropy weak solutions is delicate in general; see \cite[Remark~3.3]{2025Dedner} and the references therein.
The stability framework used in the present work yields weak--strong uniqueness rather than uniqueness within the full class of entropy weak solutions.
\end{remark}

\section{Discretization}
\label{sec:discretization}
In this section, we provide a method-of-lines discretization of \eqref{eq:convection-diffusion}.
We combine a discontinuous Galerkin (dG) spatial discretization on Cartesian meshes with time discretization by Runge--Kutta methods.
Our a~posteriori framework will be built so that it applies to methods on this level of generality.
We also state in this section the specific methods from this class that we use in our numerical experiments.

\subsection{Discontinuous Galerkin spatial discretization}
\label{subsec:dg-spatial}

Let $\mathcal{T}_h = \{K\}$ be a Cartesian mesh of $\mathbb{T}^d$.
For each element $K$, we define $h_K := \mathrm{diam}(K)$ and set $h := \max_{K \in \mathcal{T}_h} h_K$.
We denote the set of interior faces by $\mathcal{E}_h$. Since the domain is $\mathbb{T}^d$, every face is interior.
For each face $e \in \mathcal{E}_h$, let $h_e$ denote its diameter, and we fix a unit normal $\mathbf{n}_e$.
For a Cartesian cell $K = \prod_{i=1}^d I_i$, let $\mathcal{Q}_q(K)$ denote the space of tensor-product polynomials of degree at most $q$ in each coordinate.
The discontinuous Galerkin approximation space of piecewise tensor-product polynomials of degree at most $q$ is
\begin{equation}
\label{eq:dg-space}
V_q^s := \{ v \in L^2(\mathbb{T}^d) : v|_K \in \mathcal{Q}_q(K), \, \forall K \in \mathcal{T}_h \},
\end{equation}
where the superscript $s$ denotes spatial discretization and $\mathbf{V}_q^s := (V_q^s)^m$ is the corresponding vector-valued space.
Here and throughout the paper, we assume that the polynomial degree satisfies $q\ge 1$, so that we obtain a consistent discretization of the diffusion terms. If one is only interested in the hyperbolic case, $\varepsilon=0$, this assumption can be dropped.
For a face $e \in \mathcal{E}_h$ shared by two elements, we denote these elements by $K^-$ and $K^+$ such that $\mathbf{n}_e$ is oriented from $K^-$ to $K^+$. We define the traces
\begin{equation*}
\phi^{\pm}(\mathbf{x}) := \lim_{\delta \to 0^+} \phi(\mathbf{x} \pm \delta \mathbf{n}_e), \quad \mathbf{x} \in e,
\end{equation*}
and the jump and average operators 
\begin{align}
\label{eq:jump-average}
[\![\phi]\!] := \phi^+ - \phi^-, \quad
\{\!\!\{\phi\}\!\!\} := \tfrac{1}{2}(\phi^+ + \phi^-).
\end{align}
These definitions extend component-wise to vector-valued functions.

We consider an abstract semidiscrete discontinuous Galerkin scheme that seeks $\mathbf{u}_h(t) \in \mathbf{V}_q^s$ satisfying
\begin{equation}
\label{eq:semi-discrete}
\partial_t \mathbf{u}_h + \mathfrak{f}_h(\mathbf{u}_h) = \varepsilon \mathfrak{A}_h(\mathbf{u}_h),
\end{equation}
where $\mathfrak{f}_h: \mathbf{V}_q^s \to \mathbf{V}_q^s$ and $\mathfrak{A}_h: \mathbf{V}_q^s \to \mathbf{V}_q^s$ denote discrete operators approximating $\sum_{\alpha=1}^d \partial_{x_\alpha}\mathbf{f}_\alpha(\cdot)$ and $\sum_{\alpha,\beta=1}^d \partial_{x_\alpha}\bigl(\mathbf{A}_{\alpha\beta}(\cdot)\partial_{x_\beta}(\cdot)\bigr)$, respectively.

We formulate the scheme at a rather general level, since our a~posteriori analysis does not impose structural assumptions on the underlying numerical scheme. However, the scaling behavior of the resultant error estimator will reflect the properties of the error of the employed scheme.

Denoting by $(\cdot,\cdot)$ the $L^2$ inner product on $\mathbb{T}^d$, a standard discontinuous Galerkin discretization of the convective term is obtained by defining $\mathfrak{f}_h(\mathbf{u}_h) \in \mathbf{V}_q^s$ through
\begin{equation}
\label{eq:convective-dg}
(\mathfrak{f}_h(\mathbf{u}_h), \boldsymbol{\psi}_h) :=
-\sum_{K \in \mathcal{T}_h} \int_K \sum_{\alpha=1}^d \mathbf{f}_\alpha(\mathbf{u}_h)\cdot \partial_{x_\alpha}\boldsymbol{\psi}_h \, \mathrm{d}x
+ \sum_{e \in \mathcal{E}_h} \int_e \widehat{\mathbf{f}}(\mathbf{u}_h^-, \mathbf{u}_h^+; \mathbf{n}_e)
  \cdot [\![\boldsymbol{\psi}_h]\!] \, \mathrm{d}s
\quad \forall \boldsymbol{\psi}_h \in \mathbf{V}_q^s,
\end{equation}
where $\widehat{\mathbf{f}} : U \times U \times \mathbb{S}^{d-1} \to \mathbb{R}^m$ denotes a consistent numerical flux, satisfying
\begin{equation}
\label{eq:convective-flux-consistency}
\widehat{\mathbf{f}}(\mathbf{a},\mathbf{a};\mathbf{n}) = \sum_{\alpha=1}^d n_\alpha \mathbf{f}_\alpha(\mathbf{a}) \qquad \forall \mathbf{a} \in U,\ \forall \mathbf{n} \in \mathbb{S}^{d-1},
\end{equation}
see, e.g.,~\cite{2012Pietro}. A broad class of discrete diffusion operators can be treated within our framework.
For instance, one may take $\mathfrak{A}_h$ to be the interior penalty discretization of the diffusion operator defined by
\begin{align}
  \label{eq:diffusive-dg}
(\mathfrak{A}_h(\mathbf{u}_h), \boldsymbol{\psi}_h) &:=
-\sum_{K \in \mathcal{T}_h} \int_K \sum_{\alpha,\beta=1}^d \mathbf{A}_{\alpha\beta}(\mathbf{u}_h)\,\partial_{x_\beta}\mathbf{u}_h \cdot \partial_{x_\alpha}\boldsymbol{\psi}_h \, \mathrm{d}x \notag \\
&\quad + \sum_{e \in \mathcal{E}_h} \int_e \Bigl(
  \{\!\!\{\mathbf{g}_e(\mathbf{u}_h;\mathbf{u}_h)\}\!\!\} \cdot [\![\boldsymbol{\psi}_h]\!]
  + \{\!\!\{\mathbf{g}_e(\mathbf{u}_h;\boldsymbol{\psi}_h)\}\!\!\} \cdot [\![\mathbf{u}_h]\!]
  \Bigr) \, \mathrm{d}s \notag \\
&\quad - \sum_{e \in \mathcal{E}_h} \frac{1}{h_e} \int_e
  \mathbf{D}_e[\![\mathbf{u}_h]\!] \cdot [\![\boldsymbol{\psi}_h]\!] \, \mathrm{d}s,
\end{align}
where $\mathbf{D}_e \in \mathbb{R}^{m\times m}$ is a facewise penalty matrix; cf.~\cite{2006HartmannHouston} and 
\begin{equation}
\label{eq:diffusive-face-flux}
\mathbf{g}_e(\mathbf{u}_h;\mathbf{v}) := \sum_{\alpha,\beta=1}^d n_{e,\alpha}\mathbf{A}_{\alpha\beta}(\mathbf{u}_h)\partial_{x_\beta}\mathbf{v}.
\end{equation}

\subsection{Runge--Kutta time discretization}
\label{subsec:rk-temporal}

We discretize \eqref{eq:semi-discrete} in time on a partition
\begin{equation}
\label{eq:time-partition}
0 = t_0 < t_1 < \cdots < t_N = T,
\end{equation}
with time steps $\tau_n := t_n - t_{n-1}$ and $\tau := \max_{1 \le n \le N} \tau_n$. Given the initial datum $\mathbf{u}_h^0$, we seek approximations $\mathbf{u}_h^n$ for $n=1,\ldots,N$.
Our analysis accommodates general Runge--Kutta (RK) schemes, including  implicit-explicit (IMEX) methods, as well as multistep schemes.
The temporal reconstruction in Section~\ref{subsec:temporal} uses only the values $\mathbf{u}_h^n$ at the time nodes and the evaluations $-\mathfrak{f}_h(\mathbf{u}_h^n) + \varepsilon \mathfrak{A}_h(\mathbf{u}_h^n)$, and is therefore not tied to a specific time-stepping method~\cite{2016Dedner}.

The splitting in \eqref{eq:semi-discrete} naturally suggests an IMEX method, treating the convective contribution $-\mathfrak{f}_h$ explicitly and the diffusive contribution $\varepsilon \mathfrak{A}_h$ implicitly to avoid time-step restrictions of order $\varepsilon h^{-2}$. Thus, applying an $s$-stage IMEX Runge--Kutta method to \eqref{eq:semi-discrete}, we seek stage values $\mathbf U_h^{(i)}$ satisfying
\begin{equation}
\label{eq:imex-rk-stage}
\mathbf U_h^{(i)}
=
\mathbf u_h^n
-\tau_n \sum_{j=1}^{i-1}\widetilde a_{ij}\,\mathfrak f_h(\mathbf U_h^{(j)})
+\tau_n \sum_{j=1}^{i} a_{ij}\,\varepsilon \mathfrak A_h(\mathbf U_h^{(j)}),
\qquad i=1,\dots,s,
\end{equation}
and then set
\begin{equation}
\label{eq:imex-rk-update}
\mathbf u_h^{n+1}
=
\mathbf u_h^n
-\tau_n \sum_{i=1}^{s}\widetilde b_i\,\mathfrak f_h(\mathbf U_h^{(i)})
+\tau_n \sum_{i=1}^{s} b_i\,\varepsilon \mathfrak A_h(\mathbf U_h^{(i)}),
\end{equation}
where $\widetilde a_{ij}$ and $\widetilde b_i$ denote the coefficients of the explicit Butcher tableau, while $a_{ij}$ and $b_i$ denote those of the implicit tableau.
One may, for instance, consider the IMEX additive Runge--Kutta methods $\mathrm{ARK3(2)4L[2]SA}$ and $\mathrm{ARK5(4)8L[2]SA}$~\cite{2003KennedyCarpenter}; see Appendix~\ref{app:imex-rk-tableaux} for their Butcher tableaux.

\section{Space-time reconstruction}
\label{sec:reconstruction}

In this section, we describe how to obtain a sufficiently regular space--time reconstruction of the fully discrete approximation. This is necessary because the dG approximation is available only at the discrete time nodes and is discontinuous across element interfaces, whereas the stability estimates used in the present work require a more regular function.
The temporal reconstruction is the same as presented in \cite{2016Dedner} while using SIAC filtering as spatial reconstruction is the main novelty of this work.

\subsection{Temporal reconstruction}
\label{subsec:temporal}

Following~\cite{2016Dedner}, we construct a temporal reconstruction by Hermite interpolation in time.
For a fixed integer $p\in\mathbb{N}_0$, we seek a local polynomial reconstruction on $[t_n,t_{n+1}]$ using the data at $t_{n-p},\ldots,t_n,t_{n+1}$. We impose that it matches both the discrete solution value and the first time derivative at each of these $p+2$ time nodes. This gives $2(p+2)=2p+4$ interpolation conditions, and we therefore use a polynomial of degree $\ell:=2p+3$, so that the number of conditions matches the number of coefficients in the polynomial. 

For each $n=p,\ldots,N-1$, let $\widehat{\mathbf{u}}_h^n \in \mathcal{P}_\ell(\mathbb{R};\mathbf{V}_q^s)$ be the polynomial associated with the interval $[t_n,t_{n+1}]$ and satisfying
\begin{equation}
\label{eq:temporal-local-reconstruction}
\widehat{\mathbf{u}}_h^n(t_j)=\mathbf{u}_h^j,
\qquad
d_t\widehat{\mathbf{u}}_h^n(t_j)
=
-\mathfrak{f}_h(\mathbf{u}_h^j)+\varepsilon\mathfrak{A}_h(\mathbf{u}_h^j),
\qquad
j=n-p,\ldots,n+1.
\end{equation}
By standard Hermite interpolation, this polynomial is well defined.

For an arbitrary vector space $V$, we define
\begin{equation}
\label{eq:Vt-space}
\mathbb{V}^t_\ell(0,T;V) := \{ w:[0,T]\to V \mid w|_{(t_{n-1},t_n)} \in \mathcal{P}_\ell((t_{n-1},t_n),V),\ n = 1,\ldots,N \}.
\end{equation}
We then define the temporal reconstruction $\widehat{\mathbf{u}}_h^t \in \mathbb{V}^t_\ell(0,T;\mathbf{V}_q^s)$ by
\begin{equation}
\label{eq:temporal-global-reconstruction}
\left.\widehat{\mathbf{u}}_h^t\right|_{[t_n,t_{n+1}]}
:=
\left.\widehat{\mathbf{u}}_h^n\right|_{[t_n,t_{n+1}]},
\qquad n=p,\ldots,N-1.
\end{equation}
On the initial time interval, where values from the previous $p$ time levels are not available, the temporal reconstruction is defined analogously using the available data at $t_0,\ldots,t_p$. This means the error estimator for the first few time steps can only be evaluated after the $p$-th step. 
Moreover, by \cite[Lemma~2.4]{2016Dedner}, we have $\widehat{\mathbf{u}}_h^t \in W^{1,\infty}(0,T;\mathbf{V}_q^s)$.



\subsection{Spatial reconstruction by SIAC filtering}
\label{subsec:siac}

As a spatial reconstruction operator, we use SIAC filtering, a post-processing technique for discontinuous Galerkin approximations that improves smoothness while preserving (and actually increasing) high-order accuracy.
In the present work, we restrict our attention to Cartesian meshes, where the multidimensional SIAC kernel can be defined naturally as a tensor product of one-dimensional kernels~\cite{2014Ryan}.
We begin by recalling the central B-splines from which the SIAC kernels are built.

\begin{definition}[Central B-spline \cite{2014Ryan}]
\label{def:central-bspline}
The central B-spline of order $\ell \geq 1$ is defined recursively by
\begin{equation}
\label{eq:bspline-recursive}
\psi^{(1)}(x) = \chi_{[-1/2,1/2]}(x), \quad \psi^{(\ell+1)}(x) = \left(\psi^{(1)} * \psi^{(\ell)}\right)(x),
\end{equation}
where $\chi_{[-1/2,1/2]}$ is the characteristic function and $*$ denotes convolution. The order-$\ell$ B-spline $\psi^{(\ell)}$ has support $[-\ell/2, \ell/2]$, smoothness $C^{\ell-2}$, and is a piecewise polynomial of degree $\ell-1$.
\end{definition}

\begin{definition}[SIAC kernel \cite{2014Ryan}]
\label{def:siac-kernel}
In one dimension, the SIAC kernel is defined by
\begin{equation}
\label{eq:siac-kernel-1d}
K^{(2q+1,q+1)}(x) := \sum_{\gamma=0}^{2q} c_\gamma^{(2q+1,q+1)} \psi^{(q+1)}(x-(-q+\gamma)),
\end{equation}
where $\psi^{(q+1)}$ is the central B-spline of order $q+1$ and the coefficients $c_\gamma^{(2q+1,q+1)}$, $\gamma =0,\dots, 2q$, are determined by the requirement that
\begin{equation}
\label{eq:polynomial-reproduction}
K^{(2q+1,q+1)} * x^p = x^p, \qquad \text{for} \ p = 0,1,\ldots,2q.
\end{equation}
In $d$ dimensions, the SIAC kernel is defined as the tensor product of one-dimensional kernels
\begin{equation}
\label{eq:siac-kernel-multid}
K^{(2q+1,q+1)}(\mathbf{x}) := \prod_{i=1}^d K^{(2q+1,q+1)}(x_i), \qquad \mathbf{x} = (x_1,\ldots,x_d) \in \mathbb{R}^d.
\end{equation}
Here and below, we use $x \in \mathbb{R}$ for the one-dimensional variable and $\mathbf{x} \in \mathbb{R}^d$ for the multidimensional variable. Thus, the notation $K^{(2q+1,q+1)}$ refers to the one-dimensional kernel or its tensor-product extension according to its argument.
We define the scaled kernel by
\begin{equation}
\label{eq:siac-kernel-scaled}
K_h^{(2q+1,q+1)}(\mathbf{x}) := h^{-d}K^{(2q+1,q+1)}(\mathbf{x}/h).
\end{equation}
\end{definition}

\begin{definition}[SIAC filtering]
\label{def:siac-filtering}
The SIAC filtering of a function $\mathbf{v}_h \in \mathbf{V}_q^s$ is defined by
\begin{equation}
\label{eq:siac-filtering}
(K_h^{(2q+1,q+1)} * \mathbf{v}_h)(\mathbf{x}) := \int_{\mathbb{R}^d} K_h^{(2q+1,q+1)}(\mathbf{x} - \mathbf{y}) \mathbf{v}_h(\mathbf{y}) \, \mathrm{d}\mathbf{y},
\end{equation}
where $K_h^{(2q+1,q+1)}: \mathbb{R}^d \to \mathbb{R}$ denotes the scaled tensor-product SIAC kernel. Since $K_h^{(2q+1,q+1)}$ is scalar-valued, the convolution acts componentwise on $\mathbf{v}_h$. On the torus $\mathbb{T}^d$, the convolution is understood via the periodic extension of $\mathbf{v}_h$ to $\mathbb{R}^d$.
\end{definition}

\begin{remark}[Kernel support]
\label{rem:kernel-support}
For the scaled one-dimensional SIAC kernel, the support is given by
\[
\operatorname{supp} K_h^{(2q+1,q+1)}
=
\left[-\frac{(3q+1)h}{2},\,\frac{(3q+1)h}{2}\right].
\]
Hence the SIAC convolution at any point involves values of the discrete solution only from a neighborhood of width $(3q+1)h$ around that point~\cite{2014Ryan}.
\end{remark}

\begin{remark}[Piecewise polynomial structure of SIAC filtering]
\label{rem:piecewise-structure}
Given a piecewise polynomial $v_h$ with polynomial degree $q$ and the $q+1$-th order SIAC kernel, then the SIAC-filtered function is piecewise polynomial with polynomial degree at most $2q+1$ on each piece.
Let us consider a function $v_h(x)=\chi_{[0,h]}(x) p(x)$ where $\chi$ denotes a characteristic function and $p$ is a polynomial of degree at most $q$, and set
\begin{equation}
F_p(x)
:= (K_h^{(2q+1,q+1)}*v_h)(x)
= \int_{x-h}^x K_h^{(2q+1,q+1)}(z)p(x-z)\,\mathrm{d}z.
\end{equation}
Let  $\mathcal{B}_h$ denote the finite set of points at which $K_h^{(2q+1,q+1)}$ is not smooth and set
\[
S_h:=\mathcal{B}_h\cup(\mathcal{B}_h+h).
\]
At any $x \in \mathbb{R}\setminus S_h$, differentiating gives
\[
F_p'(x)=K_h^{(2q+1,q+1)}(x)p(0)-K_h^{(2q+1,q+1)}(x-h)p(h)+\int_{x-h}^x K_h^{(2q+1,q+1)}(z)p'(x-z)\,\mathrm{d}z.
\]
Repeating this calculation shows by induction that, for every $m\ge 1$,
\begin{multline}
F_p^{(m)}(x)
=
\sum_{\ell=0}^{m-1}
\Bigl(
\frac{\mathrm{d}^{\ell}}{\mathrm{d}x^{\ell}}K_h^{(2q+1,q+1)}(x)\,
p^{(m-1-\ell)}(0)
-
\frac{\mathrm{d}^{\ell}}{\mathrm{d}x^{\ell}}K_h^{(2q+1,q+1)}(x-h)\,
p^{(m-1-\ell)}(h)
\Bigr)
\notag\\
+ \int_{x-h}^x K_h^{(2q+1,q+1)}(z)\,p^{(m)}(x-z)\,\mathrm{d}z.
\end{multline}
Since $K_h^{(2q+1,q+1)}$ and $K_h^{(2q+1,q+1)}(\,\cdot-h)$ are piecewise polynomial  and $p^{(m)}\equiv0$ for $m\ge q+1$, it follows that $F_p^{(2q+2)}=0$ outside $S_h$. Thus, $F_p$ is polynomial on each open interval disjoint from $S_h$, with degree at most $2q+1$, so its polynomial representation can change only at points in $S_h$.

Figure~\ref{fig:siac-piecewise-structure} illustrates this mechanism for $v_h(y)=\chi_{[0,h]}(y)$ on an equidistant mesh in the cases $q=1$ and $q=2$. 
In particular, the values of $x$ at which the polynomial pieces of $K_h^{(2q+1,q+1)} * v_h$ change coincide with the original mesh interfaces for $q=1$, whereas for $q=2$ they coincide with the midpoints of cells of the original mesh.

What we see is that if the original mesh is equidistant, then $K_h^{(2q+1,q+1)} * v_h$ is piecewise polynomial with respect to a rather simple mesh even if $v_h$ is supported on more than one cell.
In contrast, if we start from a non-equidistant mesh, computing the mesh with respect to which $K_h^{(2q+1,q+1)} * v_h$ is piecewise polynomial is rather cumbersome.

In multiple space dimensions, the tensor-product structure of the SIAC kernel allows the same argument to be applied in each coordinate direction, so the filtered function is piecewise polynomial.

In principle, one does not need to use $K_h^{(2q+1,q+1)} * u_h$ directly in our a~posteriori analysis but may instead use an interpolation of $K_h^{(2q+1,q+1)} * u_h$ into a continuous finite element space. In that case, it is sufficient to compute the values of $K_h^{(2q+1,q+1)} * u_h$ at a finite number of points.
\end{remark}

\begin{figure}[tbp]
\centering
\includegraphics[width=0.75\textwidth]{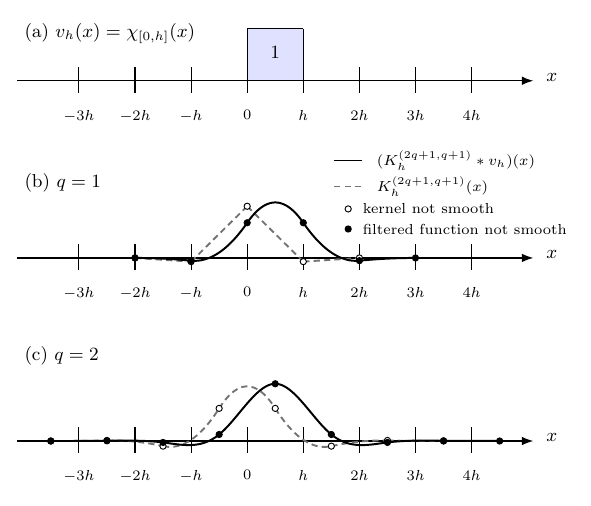}
\caption{Piecewise polynomial structure of the SIAC-filtered function for $v_h(y)=\chi_{[0,h]}(y)$ on an equidistant mesh. Panel~(a) shows $v_h$. Panels~(b) and~(c) show the SIAC kernel $K_h^{(2q+1,q+1)}$ and the SIAC-filtered function $K_h^{(2q+1,q+1)}*v_h$ for $q=1$ and $q=2$, respectively. The open circles mark the points where the kernel is not smooth, and the filled circles mark the points where the filtered function is not smooth. For $q=1$, the $x$-coordinates of the filled circles coincide with the original mesh nodes, whereas for $q=2$ they coincide with the midpoints of cells of the original mesh.}
\label{fig:siac-piecewise-structure}
\end{figure}

\begin{remark}[Superconvergence properties of SIAC filtering]
\label{rem:siac-superconvergence}
SIAC filtering can be viewed as a post-processing procedure that extracts the hidden superconvergence of discontinuous Galerkin approximations.
For linear hyperbolic problems on periodic domains, the SIAC-filtered solution is known to converge in $L^2$ with order $\mathcal{O}(h^{2q+1})$, whereas the underlying dG solution converges with order $\mathcal{O}(h^{q+1})$, provided that the exact solution is sufficiently smooth~\cite{2003Cockburn,2015Ryan}.
Analogous rigorous results are also available for linear convection--diffusion equations~\cite{2012Ji}.
For nonlinear problems, however, rigorous superconvergence results are much more limited.
Available results concern one-dimensional scalar conservation laws and certain symmetric systems of hyperbolic conservation laws~\cite{2016Meng,2017Meng}.
To the best of our knowledge, no comparably general rigorous theory is available for nonlinear multidimensional convection--diffusion problems.
In the present work, SIAC filtering is used primarily to construct a sufficiently regular spatial reconstruction for the a~posteriori analysis.
\end{remark}

Using the SIAC filter introduced above, we now define the space--time reconstruction by 
\begin{equation}
\label{eq:space-time-notation}
\widehat{\mathbf{u}}^{ts}(t, \cdot) := K_h^{(2q+1,q+1)} * \widehat{\mathbf{u}}_h^t(t, \cdot), \quad t \in (0,T).
\end{equation}
We next establish its regularity.

\begin{lemma}[Regularity of space-time reconstruction]
\label{lem:siac-regularity}
The space-time reconstruction $\widehat{\mathbf{u}}^{ts}$ satisfies
\begin{equation*}
\widehat{\mathbf{u}}^{ts}\in W^{1,\infty}((0,T)\times\mathbb{T}^d;\mathbb{R}^m).
\end{equation*}
\end{lemma}

\begin{proof}
Let $K_h := K_h^{(2q+1,q+1)}$. By \cite[Lemma~2.4]{2016Dedner}, we have $\widehat{\mathbf{u}}_h^t\in W^{1,\infty}(0,T;\mathbf{V}_q^s)$. Hence $\widehat{\mathbf{u}}_h^t,\partial_t\widehat{\mathbf{u}}_h^t\in L^\infty(0,T;\mathbf{V}_q^s)$.
Since $\mathbf{V}_q^s$ is a finite-dimensional space of piecewise polynomials on a fixed mesh, all norms on $\mathbf{V}_q^s$ are equivalent. In particular, there exists $C>0$ such that
\begin{equation*}
\|\mathbf{v}\|_{L^\infty(\mathbb{T}^d;\mathbb{R}^m)} \le C \|\mathbf{v}\|_{\mathbf{V}_q^s}
\qquad\text{for all }\mathbf{v}\in \mathbf{V}_q^s.
\end{equation*}
Therefore,
\begin{equation*}
\widehat{\mathbf{u}}_h^t,\ \partial_t\widehat{\mathbf{u}}_h^t \in L^\infty((0,T)\times\mathbb{T}^d;\mathbb{R}^m).
\end{equation*}
Since $q\ge 1$, we have $K_h\in W^{1,1}(\mathbb{R}^d)$. Hence, in the sense of distributions,
\begin{equation*}
\widehat{\mathbf{u}}^{ts}=K_h * \widehat{\mathbf{u}}_h^t,
\qquad
\partial_t\widehat{\mathbf{u}}^{ts}=K_h * \partial_t\widehat{\mathbf{u}}_h^t,
\qquad
\nabla_x\widehat{\mathbf{u}}^{ts}=(\nabla K_h)*\widehat{\mathbf{u}}_h^t.
\end{equation*}
Applying Young's inequality for convolution on $\mathbb{T}^d$, we obtain
\begin{equation*}
\|\widehat{\mathbf{u}}^{ts}\|_{L^\infty((0,T)\times\mathbb{T}^d;\mathbb{R}^m)}
\le
\|K_h\|_{L^1(\mathbb{R}^d)}
\|\widehat{\mathbf{u}}_h^t\|_{L^\infty((0,T)\times\mathbb{T}^d;\mathbb{R}^m)},
\end{equation*}
\begin{equation*}
\|\partial_t\widehat{\mathbf{u}}^{ts}\|_{L^\infty((0,T)\times\mathbb{T}^d;\mathbb{R}^m)}
\le
\|K_h\|_{L^1(\mathbb{R}^d)}
\|\partial_t\widehat{\mathbf{u}}_h^t\|_{L^\infty((0,T)\times\mathbb{T}^d;\mathbb{R}^m)},
\end{equation*}
\begin{equation*}
\|\nabla_x\widehat{\mathbf{u}}^{ts}\|_{L^\infty((0,T)\times\mathbb{T}^d;\mathbb{R}^{m\times d})}
\le
\|\nabla K_h\|_{L^1(\mathbb{R}^d;\mathbb{R}^d)}
\|\widehat{\mathbf{u}}_h^t\|_{L^\infty((0,T)\times\mathbb{T}^d;\mathbb{R}^m)}.
\end{equation*}
Hence $\widehat{\mathbf{u}}^{ts}\in W^{1,\infty}((0,T)\times\mathbb{T}^d;\mathbb{R}^m)$.
\end{proof}

\subsection{Residual splitting for the space-time reconstruction}
\label{sec:aposteriori}

In this subsection, we introduce a splitting of the residual associated with the space--time reconstruction $\widehat{\mathbf{u}}^{ts}$.
The residual is defined by
\begin{equation}
\label{eq:residual-def}
\mathbf{r} := \partial_t \widehat{\mathbf{u}}^{ts} + \sum_{\alpha=1}^{d} \partial_{x_\alpha} \mathbf{f}_\alpha(\widehat{\mathbf{u}}^{ts}) - \varepsilon \sum_{\alpha, \beta=1}^{d} \partial_{x_\alpha} \bigl( \mathbf{A}_{\alpha\beta}(\widehat{\mathbf{u}}^{ts}) \partial_{x_\beta} \widehat{\mathbf{u}}^{ts} \bigr).
\end{equation}

\begin{remark}[Residual splitting]
\label{rem:splitting}
If $\mathbf{r}$ were treated as a single term, one would either impose unnecessarily strong regularity requirements or introduce a factor of $\varepsilon^{-1}$ that blows up as $\varepsilon \to 0$.
 To be more precise, in an idealized version of the relative entropy computation, see \cite{2025Dedner} for details, the term
\begin{equation}
 \int_0^T \int_{\mathbb{T}^d}
 \mathbf{r} (\bu - \widehat{ \bu}) \mathrm{d} x\mathrm{d} t
\end{equation}
arises and needs to be estimated in terms of norms of $\mathbf{r}$ and available norms of $\bu - \widehat \bu$. The available norms are $\|\bu - \widehat \bu\|_{L^\infty(0,T;L^ 2(\mathbb{T}^d))}$ from the relative entropy and $\sqrt{\varepsilon}\|\bu - \widehat \bu\|_{L^2(0,T;H^1(\mathbb{T}^d))}$ from the dissipation.

Neither of the two following straightforward estimates is satisfactory.
On the one hand, estimating 
\begin{equation}
 \int_0^ T \int_{\mathbb{T}^d}\mathbf{r} (\bu - \widehat \bu) \mathrm{d} x\mathrm{d} t\\
 \leq \|\mathbf{r}\|_{L^1(0,T;L^ 2(\mathbb{T}^d))}
 \|\bu - \widehat \bu\|_{L^\infty(0,T;L^ 2(\mathbb{T}^d))}
\end{equation}
is suboptimal for the scaling behavior in $h$ of the terms in the residual that come from the diffusion term.
On the other hand, estimating
\begin{equation}
 \int_0^ T \int_{\mathbb{T}^d}\mathbf{r} (\bu - \widehat \bu) \mathrm{d} x\mathrm{d} t\\
 \leq \frac{1}{\sqrt{\varepsilon}}\|\mathbf{r}\|_{L^2(0,T;H^{-1}(\mathbb{T}^d))}
 \sqrt{\varepsilon}\|\bu - \widehat \bu\|_{L^2(0,T;H^1(\mathbb{T}^d))}
\end{equation}
makes the error bound blow up for $\varepsilon \rightarrow 0$.
\end{remark}

We therefore decompose the residual as
\begin{equation}
\label{eq:residual-decomposition}
\mathbf{r}=\mathbf{r}_1+\varepsilon\mathbf{r}_2,
\qquad
\mathbf{r}_1 \in L^1(0,T;L^2(\mathbb{T}^d)),
\quad
\mathbf{r}_2 \in L^2(0,T;H^{-1}(\mathbb{T}^d)),
\end{equation}
where $\mathbf{r}_1$ and $\mathbf{r}_2$ denote the hyperbolic and parabolic contributions, respectively. This splitting reflects the different regularity and scaling of the convective and diffusive terms and yields estimates that scale optimally in $h$ and remain robust in the vanishing-viscosity limit.
To achieve this splitting, we introduce an auxiliary reconstruction based on the kernel employed in \cite{2009Ryan}, which is constructed from smoother B-splines.
Following Definition~\ref{def:siac-kernel}, let $K_h^{(2q+1,q+2)}$ denote the scaled one-dimensional kernel obtained from $K_h^{(2q+1,q+1)}$ by replacing the B-spline order $q+1$ with $q+2$ in its definition.
For each coordinate direction $x_\beta$, we define the directional tensor-product kernel
\begin{equation}
\label{eq:thomee-kernel}
\widetilde{K}_{h,\beta}^{(2q+1,q+2)}(\mathbf{x})
:=
K_h^{(2q+1,q+2)}(x_\beta)
\prod_{\substack{i=1\\ i\neq \beta}}^{d} K_h^{(2q+1,q+1)}(x_i),
\qquad \mathbf{x} = (x_1,\ldots,x_d)\in\mathbb{R}^d,
\end{equation}
and define the associated auxiliary reconstruction by 
\begin{equation}
\label{eq:thomee-aux-reconstruction}
\widetilde{\mathbf{u}}^{ts,\beta}
:=
\widetilde{K}_{h,\beta}^{(2q+1,q+2)} * \widehat{\mathbf{u}}_h^t,
\qquad \beta = 1,\ldots,d.
\end{equation}

We now define the following two diffusive fluxes
\begin{align}
\label{eq:diffusive-fluxes}
\widehat{\mathbf{G}}_\alpha
&:=
\sum_{\beta=1}^{d}
\mathbf{A}_{\alpha\beta}(\widehat{\mathbf{u}}^{ts}) \, \partial_{x_\beta}\widehat{\mathbf{u}}^{ts},
\\
\widetilde{\mathbf{G}}_\alpha
&:=
\sum_{\beta=1}^{d}
\mathbf{A}_{\alpha\beta}(\widehat{\mathbf{u}}^{ts}) \, \partial_{x_\beta}\widetilde{\mathbf{u}}^{ts,\beta},
\qquad \alpha = 1,\ldots,d.
\end{align}
With these fluxes at hand, we define 
\begin{align}
\label{eq:r1-def}
\mathbf{r}_1 &:= \partial_t \widehat{\mathbf{u}}^{ts} + \sum_{\alpha=1}^{d} \partial_{x_\alpha} \mathbf{f}_\alpha(\widehat{\mathbf{u}}^{ts}) - \varepsilon \sum_{\alpha=1}^{d} \partial_{x_\alpha} \widetilde{\mathbf{G}}_\alpha, \\
\label{eq:r2-def}
\mathbf{r}_2 &:= \sum_{\alpha=1}^{d} \partial_{x_\alpha} \bigl(\widetilde{\mathbf{G}}_\alpha - \widehat{\mathbf{G}}_\alpha\bigr),
\end{align}
so that the residual admits the decomposition
\begin{equation}
\label{eq:residual-split}
\mathbf{r} = \mathbf{r}_1 + \varepsilon \mathbf{r}_2.
\end{equation}
Moreover, \eqref{eq:r2-def} immediately yields, by duality and the Cauchy--Schwarz inequality,
\begin{equation}
\label{eq:r2-H-1-bound}
\|\mathbf{r}_2\|_{L^2(0,T; H^{-1}(\mathbb{T}^d;\mathbb{R}^m))}
\le \|\widetilde{\mathbf{G}} - \widehat{\mathbf{G}}\|_{L^2((0,T)\times \mathbb{T}^d;\mathbb{R}^{m\times d})}
=: \CE_{r_2}.
\end{equation}

\section{Stability estimates and a~posteriori analysis}
\label{sec:relative-entropy}

In this section, we state  stability estimates on the periodic domain $\mathbb{T}^d$ that underlie our a~posteriori analysis.
They are multidimensional versions of estimates in \cite{2025Dedner} where a detailed stability analysis for one dimensional problems of the form \eqref{eq:convection-diffusion} was carried out.
The results stated below cover the problem classes considered in our numerical experiments.
We omit the proofs, since they are analogous to the one-dimensional proofs in~\cite{2025Dedner}.

\subsection{Linear advection--diffusion equation}

We begin with the linear advection--diffusion equation
\begin{equation}
\label{eq:lin-ad-diff-scalar}
\partial_t u + \nabla\cdot(\mathbf{a}u)
= \varepsilon\sum_{\alpha=1}^d \partial_{x_\alpha x_\alpha}u
\end{equation}
on $(0,T)\times\mathbb{T}^d$, where $\mathbf{a}\in\mathbb{R}^d$ is constant and $\varepsilon>0$.
For this problem, the entropy compatibility condition \eqref{eq:positive-semidefinite} is immediate from the quadratic entropy $\eta(u)=\tfrac12 u^2$.
We also consider the perturbed equation
\begin{equation}
\label{eq:lin-ad-diff-scalar-pert}
\partial_t \widehat{u} + \nabla\cdot(\mathbf{a}\widehat{u})
= \varepsilon\sum_{\alpha=1}^d \partial_{x_\alpha x_\alpha}\widehat{u} + r,
\qquad r=r_1+\varepsilon r_2,
\end{equation}
and state the corresponding multidimensional analogue of \cite[Theorem~4.4]{2025Dedner}.

\begin{proposition}[Linear advection--diffusion equation]
\label{thm:linear-scalar-stability}
Let $u\in L^2(0,T;H^1(\mathbb{T}^d))$ be an entropy weak solution of \eqref{eq:lin-ad-diff-scalar} with initial data $u_0\in H^1(\mathbb{T}^d)$.
Let $\widehat{u}\in W^{1,\infty}((0,T)\times\mathbb{T}^d)$ be a Lipschitz continuous weak solution of \eqref{eq:lin-ad-diff-scalar-pert} with initial data $\widehat{u}(0,\cdot)\in W^{1,\infty}(\mathbb{T}^d)$, where $r_1\in L^1(0,T;L^2(\mathbb{T}^d))$ and $r_2\in L^2(0,T;H^{-1}(\mathbb{T}^d))$.
Then, for almost all $t\in(0,T]$,
\begin{equation}
\label{eq:linear-scalar-stability}
\begin{aligned}
\|u-\widehat{u}\|_{L^\infty(0,t;L^2(\mathbb{T}^d))}^2
+\varepsilon |u-\widehat{u}|_{L^2(0,t;H^1(\mathbb{T}^d))}^2
&\le
2\|u_0-\widehat{u}(0)\|_{L^2(\mathbb{T}^d)}^2
\\
&{}+4\|r_1\|_{L^1(0,t;L^2(\mathbb{T}^d))}^2
+2\varepsilon\|r_2\|_{L^2(0,t;H^{-1}(\mathbb{T}^d))}^2.
\end{aligned}
\end{equation}
\end{proposition}

\subsection{Nonlinear scalar equation}

We next consider nonlinear scalar convection--diffusion equations
\begin{equation}
\label{eq:nonlinear-scalar}
\partial_t u + \sum_{\alpha=1}^d \partial_{x_\alpha} f_\alpha(u)
= \varepsilon\sum_{\alpha=1}^d \partial_{x_\alpha}\bigl(A(u)\partial_{x_\alpha}u\bigr),
\end{equation}
on $(0,T)\times\mathbb{T}^d$, where $\mathbf{f}\in C^2(U;\mathbb{R}^d)$, $A\in C^2(U;\mathbb{R})$, and $\varepsilon>0$.
In the scalar case, the entropy compatibility condition \eqref{eq:positive-semidefinite} reduces to the nonnegativity of the scalar diffusion coefficient.
In particular, it is satisfied for the viscous Burgers equation considered in the numerical experiments, with the quadratic entropy $\eta(u)=\tfrac12 u^2$.
We also consider the perturbed equation
\begin{equation}
\label{eq:nonlinear-scalar-pert}
\partial_t \widehat{u} + \sum_{\alpha=1}^d \partial_{x_\alpha} f_\alpha(\widehat{u})
= \varepsilon\sum_{\alpha=1}^d \partial_{x_\alpha}\bigl(A(\widehat{u})\partial_{x_\alpha}\widehat{u}\bigr) + r,
\qquad r=r_1+\varepsilon r_2,
\end{equation}
and state the corresponding multidimensional analogue of \cite[Theorem~5.3]{2025Dedner}.

\begin{proposition}[Nonlinear scalar equation]
\label{thm:nonlinear-scalar-stability}
Let $u\in L^2(0,T;H^1(\mathbb{T}^d))\cap L^\infty((0,T)\times\mathbb{T}^d)$ be an entropy weak solution of \eqref{eq:nonlinear-scalar} with initial data $u_0\in H^1(\mathbb{T}^d)$.
Let $\widehat{u}\in W^{1,\infty}((0,T)\times\mathbb{T}^d)$ be a Lipschitz continuous weak solution of \eqref{eq:nonlinear-scalar-pert} with initial data $\widehat{u}(0,\cdot)\in W^{1,\infty}(\mathbb{T}^d)$, where $r_1\in L^1(0,T;L^2(\mathbb{T}^d))$ and $r_2\in L^2(0,T;H^{-1}(\mathbb{T}^d))$.
Assume that there exists a compact set $K\subset U$ such that
\[
u,\ \widehat{u}\in K
\qquad \text{a.e. on } (0,T)\times\mathbb{T}^d,
\]
and, after normalization,
\[
\inf_{z\in K} A(z)=1.
\]
Then, for almost all $t\in(0,T]$,
\begin{equation}
\label{eq:nonlinear-scalar-stability}
\begin{aligned}
\|u-\widehat{u}\|_{L^\infty(0,t;L^2(\mathbb{T}^d))}^2
+\varepsilon|u-\widehat{u}|_{L^2(0,t;H^1(\mathbb{T}^d))}^2
&\le
\Bigl(
4\|u_0-\widehat{u}(0)\|_{L^2(\mathbb{T}^d)}^2\Bigr. \\
&\qquad
\Bigl. +16\|r_1\|_{L^1(0,t;L^2(\mathbb{T}^d))}^2
+8\varepsilon\|r_2\|_{L^2(0,t;H^{-1}(\mathbb{T}^d))}^2
\Bigr)e^{2\Lambda t},
\end{aligned}
\end{equation}
where $\Lambda$ is given by
\begin{equation}
\label{eq:K-nonlinear-scalar}
\Lambda:=
2\sum_{\alpha=1}^d
\sup_{z\in K}|f_\alpha''(z)|\,
\|\partial_{x_\alpha}\widehat{u}\|_{L^\infty((0,T)\times\mathbb{T}^d)}
+
2\varepsilon
\sup_{z\in K}|A'(z)|^2\,
\|\nabla \widehat{u}\|_{L^\infty((0,T)\times\mathbb{T}^d)}^2.
\end{equation}
\end{proposition}

\subsection[Diffusive p-system]{Diffusive \texorpdfstring{$p$}{p}-system}

We finally consider the nonlinear diffusive $p$-system
\begin{equation}
\label{eq:degenerate-wave}
\partial_t u - \sum_{\alpha=1}^d \partial_{x_\alpha} v_\alpha = 0,
\qquad
\partial_t v_\beta - \partial_{x_\beta} W'(u) = \varepsilon\sum_{\alpha=1}^d \partial_{x_\alpha x_\alpha}v_\beta,
\qquad \beta=1,\dots,d,
\end{equation}
on $(0,T)\times\mathbb{T}^d$, where $W\in C^3(U)$ is strictly convex and $\varepsilon>0$.
For this system, the entropy compatibility condition \eqref{eq:positive-semidefinite} is satisfied for the standard mechanical entropy $\eta(u,\mathbf{v})=W(u)+\tfrac12 |\mathbf{v}|^2$.
Indeed, the diffusion tensor vanishes in the $u$-component and is positive semidefinite on the velocity block, so condition \eqref{eq:positive-semidefinite} reduces to a sum of squares of the velocity components.
We also consider the perturbed system
\begin{equation}
\label{eq:degenerate-wave-pert}
\partial_t \widehat{u} - \sum_{\alpha=1}^d \partial_{x_\alpha} \widehat{v}_\alpha = r_u,
\qquad
\partial_t \widehat{v}_\beta - \partial_{x_\beta} W'(\widehat{u}) = \varepsilon\sum_{\alpha=1}^d \partial_{x_\alpha x_\alpha}\widehat{v}_\beta + (\mathbf{r}_v)_\beta,
\qquad \beta=1,\dots,d,
\end{equation}
and state the corresponding multidimensional analogue of \cite[Theorem~8.4]{2025Dedner}.

\begin{proposition}[Diffusive $p$-system]
\label{thm:nonlinear-hyperbolic-parabolic-stability}
Let $(u, \mathbf{v})$ with $u\in L^\infty((0,T)\times\mathbb{T}^d)$ and $\mathbf{v}\in L^2(0,T;H^1(\mathbb{T}^d;\mathbb{R}^d))$ be an entropy weak solution of \eqref{eq:degenerate-wave} with initial data $u_0\in H^1(\mathbb{T}^d)$ and $\mathbf{v}_0\in H^1(\mathbb{T}^d;\mathbb{R}^d)$.
Let $(\widehat{u},\widehat{\mathbf{v}})\in W^{1,\infty}((0,T)\times\mathbb{T}^d;\mathbb{R}^{1+d})$ be a Lipschitz continuous weak solution of \eqref{eq:degenerate-wave-pert} with initial data $\widehat{u}(0,\cdot)\in W^{1,\infty}(\mathbb{T}^d)$ and $\widehat{\mathbf{v}}(0,\cdot)\in W^{1,\infty}(\mathbb{T}^d;\mathbb{R}^d)$.
Here $r_u$ is a scalar residual and $\mathbf{r}_v$ is a vector-valued residual.
Assume further that
\begin{equation*}
r_u\in L^1(0,T;L^2(\mathbb{T}^d)),
\end{equation*}
and that $\mathbf{r}_v$ can be decomposed as
\begin{equation*}
\mathbf{r}_v=\mathbf{r}_{v,1}+\varepsilon \mathbf{r}_{v,2},
\end{equation*}
with
\begin{equation*}
\mathbf{r}_{v,1}\in L^1(0,T;L^2(\mathbb{T}^d;\mathbb{R}^d)),
\qquad
\mathbf{r}_{v,2}\in L^2(0,T;H^{-1}(\mathbb{T}^d;\mathbb{R}^d)).
\end{equation*}
Assume also that there exists a compact set $K\subset U$ such that
\[
u,\ \widehat{u}\in K
\qquad \text{a.e. on } (0,T) \times \mathbb{T}^d,
\]
and define
\begin{equation}
\label{eq:cw-cw}
c_W:=\min_{z\in K} W''(z),\qquad
C_W:=\max_{z\in K} |W'''(z)|.
\end{equation}
Finally set
\begin{equation}
\label{eq:K-nonlinear-wave}
\Lambda
:=
\frac{2C_W}{c_W}
\|\operatorname{div}\widehat{\mathbf{v}}\|_{L^\infty((0,T)\times\mathbb{T}^d)}.
\end{equation}
Then, for almost all $t\in(0,T]$,

\begin{equation}
\label{eq:nonlinear-hyperbolic-parabolic-stability}
\begin{aligned}
&c_W \|u-\widehat{u}\|_{L^\infty(0,t;L^2(\mathbb{T}^d))}^2
{}+\|\mathbf{v}-\widehat{\mathbf{v}}\|_{L^\infty(0,t;L^2(\mathbb{T}^d))}^2
{}+\varepsilon |\mathbf{v}-\widehat{\mathbf{v}}|_{L^2(0,t;H^1(\mathbb{T}^d))}^2 \\
&\le \Bigl(
\frac{16}{3}\|W(u_0|\widehat{u}(0))\|_{L^1(\mathbb{T}^d)}
+\frac{8}{3}\|\mathbf{v}_0-\widehat{\mathbf{v}}(0)\|_{L^2(\mathbb{T}^d)}^2
+\frac{16}{3c_W}\|W''(\widehat{u})\,r_u\|_{L^1(0,t;L^2(\mathbb{T}^d))}^2 \\
&{}+\frac{16}{3}\|\mathbf{r}_{v,1}\|_{L^1(0,t;L^2(\mathbb{T}^d))}^2
+\frac{16}{3}\varepsilon\|\mathbf{r}_{v,2}\|_{L^2(0,t;H^{-1}(\mathbb{T}^d))}^2
\Bigr)e^{\Lambda t},
\end{aligned}
\end{equation}
where $W(u|v):=W(u)-W(v)-W'(v)(u-v)$.
\end{proposition}

\begin{remark}[A posteriori bounds for the space--time reconstruction]
\label{rem:application-spacetime}
The preceding stability estimates apply to the space--time reconstruction by taking $\widehat{\mathbf{u}}=\widehat{\mathbf{u}}^{ts}$.
By Lemma~\ref{lem:siac-regularity}, the reconstruction belongs to $W^{1,\infty}((0,T)\times\mathbb{T}^d)$, and Section~\ref{sec:aposteriori} provides the residual splitting required in the hypotheses of the propositions above.
Therefore, these stability estimates yield a~posteriori error bounds for $\mathbf{u}-\widehat{\mathbf{u}}^{ts}$ in terms of the residuals.
Note that the reconstruction is piecewise polynomial so that, given \eqref{eq:r2-H-1-bound}, the norms of the parts of the residual can be computed explicitly.
\end{remark}

\section{Numerical experiments}
\label{sec:numerical}


In this section, we investigate the scaling behavior of the a~posteriori error estimates in Section~\ref{sec:relative-entropy}.
We note that the error estimators are proportional to the sum of the approximation error for the initial data, which behaves very predictably, and the residual terms. Thus, we focus on the residual terms in order to understand the scaling behavior of the error estimators. In particular, we compare them with the error of the space--time reconstruction.
All experiments were carried out using Python, and the source code is available at \url{https://github.com/kwkwon13/a-posteriori-conv-diff-siac}.

We implement the fully discrete RKdG scheme introduced in Section~\ref{sec:discretization} for the convection--diffusion problems \eqref{eq:convection-diffusion} on periodic domains in two space dimensions.
Specifically, for the spatial discretization we use the convective discretization \eqref{eq:convective-dg} with a local Lax--Friedrichs (Rusanov) numerical flux and the symmetric interior penalty discretization \eqref{eq:diffusive-dg}. Time integration is performed with Kennedy--Carpenter implicit--explicit additive Runge--Kutta methods~\cite{2003KennedyCarpenter}; the corresponding Butcher tableaux are given in Appendix~\ref{app:imex-rk-tableaux}.

For the numerical experiments, we use uniform $N\times N$ Cartesian meshes with $N\in \{16,32,64,128\}$ and viscosities $\varepsilon\in\{0,10^{-4},10^{-3},10^{-2},10^{-1}\}$. We consider the polynomial degrees $q=1,2$ for the dG discretization and employ third-order time stepping for $q=1$ and fifth-order time stepping for $q=2$ in order to match the expected order of superconvergence of SIAC filtering; see Remark~\ref{rem:siac-superconvergence}. We set $h=1/N$ and choose the time step according to
\begin{equation*}
\Delta t
=
C_{\mathrm{adv}}\,\frac{h}{(2q+1)\lambda_{\max}} \qquad \text{with} \quad C_{\mathrm{adv}}=0.1,
\end{equation*}
where $\lambda_{\max}$ denotes the maximum advective characteristic speed in the problem under consideration.
This time-step choice is used for all values of $\varepsilon$ considered in the experiments. Since diffusion is treated implicitly, no additional diffusive time step restriction is imposed here.

Finally, we define the experimental order of convergence (EoC) for two successive meshes with mesh widths $h_1$ and $h_2$, and corresponding quantities $E_1$ and $E_2$ by
\begin{equation}
\label{eq:numerical-eoc}
\mathrm{EoC}:=\frac{\log(E_1/E_2)}{\log(h_1/h_2)}.
\end{equation}

\subsection{Linear advection--diffusion equation}
\label{subsec:numerical-linear-advection-diffusion-2d}

We consider the linear advection--diffusion equation
\begin{equation*}
\partial_t u + \nabla\cdot(\mathbf{a}u) = \varepsilon \Delta u
\end{equation*}
on $(0,T)\times\mathbb{T}^2$ with $\mathbf{a}=(1,0.5)$ and $T=1$. For this problem, $\lambda_{\max}=1$.
The exact solution is given by
\begin{equation*}
u(t,x,y)=e^{-8\varepsilon\pi^2 t}\sin\bigl(2\pi(x-t)\bigr)\cos\bigl(2\pi(y-\tfrac12 t)\bigr),
\end{equation*}
and the initial condition is chosen as $u(0,x,y)$.

Tables~\ref{tab:advdiff-2d-q1-values-eoc} and \ref{tab:advdiff-2d-q2-values-eoc} present the values and EoCs of $\|\widehat{u}_h^t-u\|_{L^2(0,T;L^2(\mathbb{T}^d))}$, $\|\widehat{u}^{ts}-u\|_{L^\infty(0,T;L^2(\mathbb{T}^d))}$, $\|r_1\|_{L^1(0,T;L^2(\mathbb{T}^d))}$, $|\widehat{u}^{ts}-u|_{L^2(0,T;H^1(\mathbb{T}^d))}$, and $\CE_{r_2}$ defined in \eqref{eq:r2-H-1-bound} for $q=1$ and $q=2$.
We observe that the time reconstruction of the dG solution converges with order $q+1$.
The space--time reconstruction exhibits superconvergence, with observed order around $2q+1$ for the smaller viscosities and dropping toward $2q$ for the larger viscosities.
In particular, the tables show that $r_1$ is of the same order as $\|\widehat{u}^{ts}-u\|_{L^\infty(0,T;L^2(\mathbb{T}^d))}$.
Moreover, $\CE_{r_2}$ is of the same order as $|\widehat{u}^{ts}-u|_{L^2(0,T;H^1(\mathbb{T}^d))}$ at the smaller viscosities and of higher order for the larger viscosities.
The rows with $\varepsilon=0$ show that the scaling behavior persists in the vanishing-viscosity limit.
This provides numerical evidence that the estimator remains robust as $\varepsilon\to 0$ and can actually be used for hyperbolic conservation laws.
For $\varepsilon=0$, the columns $|\widehat{u}^{ts}-u|_{L^2(0,T;H^1(\mathbb{T}^d))}$ and $\CE_{r_2}$ are left as \texttt{--}, since they represent the diffusion-dependent part of the estimator and therefore disappear in the inviscid case.

{
\renewcommand{\arraystretch}{1.08}
\setlength{\tabcolsep}{3.1pt}
\setlength{\LTleft}{\fill}
\setlength{\LTright}{\fill}

\begin{longtable}{ccrrrrrrrrrr}
\caption{Linear advection--diffusion equation for polynomial degree $q=1$: values and EoCs of $\|\widehat{u}_h^t-u\|_{L^2(0,T;L^2(\mathbb{T}^d))}$, $\|\widehat{u}^{ts}-u\|_{L^\infty(0,T;L^2(\mathbb{T}^d))}$, $\|r_1\|_{L^1(0,T;L^2(\mathbb{T}^d))}$, $|\widehat{u}^{ts}-u|_{L^2(0,T;H^1(\mathbb{T}^d))}$, and $\CE_{r_2}$.}
\label{tab:advdiff-2d-q1-values-eoc}\\
\toprule
\multirow{2}{*}{$\varepsilon$} & \multirow{2}{*}{$N$}
& \multicolumn{2}{c}{\begin{tabular}[c]{@{}c@{}}$\|\widehat{u}_h^t-u\|$\\$\scriptstyle{}_{L^2(0,T;L^2(\mathbb{T}^d))}$\end{tabular}}
& \multicolumn{2}{c}{\begin{tabular}[c]{@{}c@{}}$\|\widehat{u}^{ts}-u\|$\\$\scriptstyle{}_{L^\infty(0,T;L^2(\mathbb{T}^d))}$\end{tabular}}
& \multicolumn{2}{c}{\begin{tabular}[c]{@{}c@{}}$\|r_1\|$\\$\scriptstyle{}_{L^1(0,T;L^2(\mathbb{T}^d))}$\end{tabular}}
& \multicolumn{2}{c}{\begin{tabular}[c]{@{}c@{}}$|\widehat{u}^{ts}-u|$\\$\scriptstyle{}_{L^2(0,T;H^1(\mathbb{T}^d))}$\end{tabular}}
& \multicolumn{2}{c}{$\CE_{r_2}$} \\
\cmidrule(lr){3-4}\cmidrule(lr){5-6}\cmidrule(lr){7-8}\cmidrule(lr){9-10}\cmidrule(lr){11-12}
& & \multicolumn{1}{c}{Value} & \multicolumn{1}{c}{EoC} & \multicolumn{1}{c}{Value} & \multicolumn{1}{c}{EoC} & \multicolumn{1}{c}{Value} & \multicolumn{1}{c}{EoC} & \multicolumn{1}{c}{Value} & \multicolumn{1}{c}{EoC} & \multicolumn{1}{c}{Value} & \multicolumn{1}{c}{EoC} \\
\midrule
\endfirsthead
\toprule
\multirow{2}{*}{$\varepsilon$} & \multirow{2}{*}{$N$}
& \multicolumn{2}{c}{\begin{tabular}[c]{@{}c@{}}$\|\widehat{u}_h^t-u\|$\\$\scriptstyle{}_{L^2(0,T;L^2(\mathbb{T}^d))}$\end{tabular}}
& \multicolumn{2}{c}{\begin{tabular}[c]{@{}c@{}}$\|\widehat{u}^{ts}-u\|$\\$\scriptstyle{}_{L^\infty(0,T;L^2(\mathbb{T}^d))}$\end{tabular}}
& \multicolumn{2}{c}{\begin{tabular}[c]{@{}c@{}}$\|r_1\|$\\$\scriptstyle{}_{L^1(0,T;L^2(\mathbb{T}^d))}$\end{tabular}}
& \multicolumn{2}{c}{\begin{tabular}[c]{@{}c@{}}$|\widehat{u}^{ts}-u|$\\$\scriptstyle{}_{L^2(0,T;H^1(\mathbb{T}^d))}$\end{tabular}}
& \multicolumn{2}{c}{$\CE_{r_2}$} \\
\cmidrule(lr){3-4}\cmidrule(lr){5-6}\cmidrule(lr){7-8}\cmidrule(lr){9-10}\cmidrule(lr){11-12}
& & \multicolumn{1}{c}{Value} & \multicolumn{1}{c}{EoC} & \multicolumn{1}{c}{Value} & \multicolumn{1}{c}{EoC} & \multicolumn{1}{c}{Value} & \multicolumn{1}{c}{EoC} & \multicolumn{1}{c}{Value} & \multicolumn{1}{c}{EoC} & \multicolumn{1}{c}{Value} & \multicolumn{1}{c}{EoC} \\
\midrule
\endhead
\midrule
\endfoot
\bottomrule
\endlastfoot
$0$ & 16 & 6.834e-03 & -- & 4.118e-03 & -- & 4.231e-03 & -- & -- & -- & -- & -- \\
$0$ & 32 & 1.670e-03 & 2.033 & 5.076e-04 & 3.020 & 5.387e-04 & 2.973 & -- & -- & -- & -- \\
$0$ & 64 & 4.151e-04 & 2.008 & 6.279e-05 & 3.015 & 6.777e-05 & 2.991 & -- & -- & -- & -- \\
$0$ & 128 & 1.037e-04 & 2.001 & 7.884e-06 & 2.993 & 8.583e-06 & 2.981 & -- & -- & -- & -- \\
\midrule
$10^{-4}$ & 16 & 6.688e-03 & -- & 4.036e-03 & -- & 4.162e-03 & -- & 2.149e-02 & -- & 2.321e-03 & -- \\
$10^{-4}$ & 32 & 1.602e-03 & 2.062 & 4.897e-04 & 3.043 & 5.216e-04 & 2.996 & 2.570e-03 & 3.064 & 2.773e-04 & 3.065 \\
$10^{-4}$ & 64 & 3.845e-04 & 2.059 & 5.904e-05 & 3.052 & 6.396e-05 & 3.028 & 3.077e-04 & 3.062 & 3.338e-05 & 3.054 \\
$10^{-4}$ & 128 & 9.061e-05 & 2.085 & 7.164e-06 & 3.043 & 7.822e-06 & 3.032 & 3.719e-05 & 3.048 & 3.969e-06 & 3.072 \\
\midrule
$10^{-3}$ & 16 & 5.737e-03 & -- & 3.530e-03 & -- & 3.757e-03 & -- & 1.924e-02 & -- & 2.055e-03 & -- \\
$10^{-3}$ & 32 & 1.268e-03 & 2.177 & 4.241e-04 & 3.057 & 4.624e-04 & 3.023 & 2.273e-03 & 3.081 & 2.286e-04 & 3.168 \\
$10^{-3}$ & 64 & 2.846e-04 & 2.156 & 5.728e-05 & 2.889 & 6.222e-05 & 2.894 & 3.036e-04 & 2.905 & 2.606e-05 & 3.133 \\
$10^{-3}$ & 128 & 6.564e-05 & 2.116 & 9.679e-06 & 2.565 & 1.031e-05 & 2.593 & 5.092e-05 & 2.576 & 3.059e-06 & 3.091 \\
\midrule
$10^{-2}$ & 16 & 3.645e-03 & -- & 2.900e-03 & -- & 4.435e-03 & -- & 1.928e-02 & -- & 1.280e-03 & -- \\
$10^{-2}$ & 32 & 8.560e-04 & 2.090 & 6.179e-04 & 2.230 & 9.427e-04 & 2.234 & 4.004e-03 & 2.268 & 1.439e-04 & 3.153 \\
$10^{-2}$ & 64 & 2.096e-04 & 2.030 & 1.468e-04 & 2.073 & 2.237e-04 & 2.075 & 9.440e-04 & 2.085 & 1.737e-05 & 3.050 \\
$10^{-2}$ & 128 & 5.211e-05 & 2.008 & 3.624e-05 & 2.019 & 5.519e-05 & 2.019 & 2.325e-04 & 2.022 & 2.150e-06 & 3.014 \\
\midrule
$10^{-1}$ & 16 & 1.530e-03 & -- & 2.454e-03 & -- & 6.552e-03 & -- & 1.052e-02 & -- & 4.412e-04 & -- \\
$10^{-1}$ & 32 & 3.832e-04 & 1.997 & 5.966e-04 & 2.040 & 1.614e-03 & 2.021 & 2.563e-03 & 2.038 & 5.045e-05 & 3.128 \\
$10^{-1}$ & 64 & 9.588e-05 & 1.999 & 1.481e-04 & 2.010 & 4.021e-04 & 2.005 & 6.366e-04 & 2.009 & 6.141e-06 & 3.038 \\
$10^{-1}$ & 128 & 2.399e-05 & 1.999 & 3.701e-05 & 2.001 & 1.006e-04 & 2.000 & 1.591e-04 & 2.001 & 7.623e-07 & 3.010 \\
\end{longtable}
}

{
\renewcommand{\arraystretch}{1.08}
\setlength{\tabcolsep}{3.1pt}
\setlength{\LTleft}{\fill}
\setlength{\LTright}{\fill}

\begin{longtable}{ccrrrrrrrrrr}
\caption{Linear advection--diffusion equation for polynomial degree $q=2$: values and EoCs of $\|\widehat{u}_h^t-u\|_{L^2(0,T;L^2(\mathbb{T}^d))}$, $\|\widehat{u}^{ts}-u\|_{L^\infty(0,T;L^2(\mathbb{T}^d))}$, $\|r_1\|_{L^1(0,T;L^2(\mathbb{T}^d))}$, $|\widehat{u}^{ts}-u|_{L^2(0,T;H^1(\mathbb{T}^d))}$, and $\CE_{r_2}$.}
\label{tab:advdiff-2d-q2-values-eoc}\\
\toprule
\multirow{2}{*}{$\varepsilon$} & \multirow{2}{*}{$N$}
& \multicolumn{2}{c}{\begin{tabular}[c]{@{}c@{}}$\|\widehat{u}_h^t-u\|$\\$\scriptstyle{}_{L^2(0,T;L^2(\mathbb{T}^d))}$\end{tabular}}
& \multicolumn{2}{c}{\begin{tabular}[c]{@{}c@{}}$\|\widehat{u}^{ts}-u\|$\\$\scriptstyle{}_{L^\infty(0,T;L^2(\mathbb{T}^d))}$\end{tabular}}
& \multicolumn{2}{c}{\begin{tabular}[c]{@{}c@{}}$\|r_1\|$\\$\scriptstyle{}_{L^1(0,T;L^2(\mathbb{T}^d))}$\end{tabular}}
& \multicolumn{2}{c}{\begin{tabular}[c]{@{}c@{}}$|\widehat{u}^{ts}-u|$\\$\scriptstyle{}_{L^2(0,T;H^1(\mathbb{T}^d))}$\end{tabular}}
& \multicolumn{2}{c}{$\CE_{r_2}$} \\
\cmidrule(lr){3-4}\cmidrule(lr){5-6}\cmidrule(lr){7-8}\cmidrule(lr){9-10}\cmidrule(lr){11-12}
& & \multicolumn{1}{c}{Value} & \multicolumn{1}{c}{EoC} & \multicolumn{1}{c}{Value} & \multicolumn{1}{c}{EoC} & \multicolumn{1}{c}{Value} & \multicolumn{1}{c}{EoC} & \multicolumn{1}{c}{Value} & \multicolumn{1}{c}{EoC} & \multicolumn{1}{c}{Value} & \multicolumn{1}{c}{EoC} \\
\midrule
\endfirsthead
\toprule
\multirow{2}{*}{$\varepsilon$} & \multirow{2}{*}{$N$}
& \multicolumn{2}{c}{\begin{tabular}[c]{@{}c@{}}$\|\widehat{u}_h^t-u\|$\\$\scriptstyle{}_{L^2(0,T;L^2(\mathbb{T}^d))}$\end{tabular}}
& \multicolumn{2}{c}{\begin{tabular}[c]{@{}c@{}}$\|\widehat{u}^{ts}-u\|$\\$\scriptstyle{}_{L^\infty(0,T;L^2(\mathbb{T}^d))}$\end{tabular}}
& \multicolumn{2}{c}{\begin{tabular}[c]{@{}c@{}}$\|r_1\|$\\$\scriptstyle{}_{L^1(0,T;L^2(\mathbb{T}^d))}$\end{tabular}}
& \multicolumn{2}{c}{\begin{tabular}[c]{@{}c@{}}$|\widehat{u}^{ts}-u|$\\$\scriptstyle{}_{L^2(0,T;H^1(\mathbb{T}^d))}$\end{tabular}}
& \multicolumn{2}{c}{$\CE_{r_2}$} \\
\cmidrule(lr){3-4}\cmidrule(lr){5-6}\cmidrule(lr){7-8}\cmidrule(lr){9-10}\cmidrule(lr){11-12}
& & \multicolumn{1}{c}{Value} & \multicolumn{1}{c}{EoC} & \multicolumn{1}{c}{Value} & \multicolumn{1}{c}{EoC} & \multicolumn{1}{c}{Value} & \multicolumn{1}{c}{EoC} & \multicolumn{1}{c}{Value} & \multicolumn{1}{c}{EoC} & \multicolumn{1}{c}{Value} & \multicolumn{1}{c}{EoC} \\
\midrule
\endhead
\midrule
\endfoot
\bottomrule
\endlastfoot
$0$ & 16 & 2.090e-04 & -- & 1.791e-05 & -- & 6.933e-06 & -- & -- & -- & -- & -- \\
$0$ & 32 & 2.612e-05 & 3.000 & 3.799e-07 & 5.559 & 2.185e-07 & 4.988 & -- & -- & -- & -- \\
$0$ & 64 & 3.265e-06 & 3.000 & 8.940e-09 & 5.409 & 6.850e-09 & 4.995 & -- & -- & -- & -- \\
$0$ & 128 & 4.081e-07 & 3.000 & 2.330e-10 & 5.262 & 2.146e-10 & 4.996 & -- & -- & -- & -- \\
\midrule
$10^{-4}$ & 16 & 2.124e-04 & -- & 1.796e-05 & -- & 7.175e-06 & -- & 1.335e-04 & -- & 3.280e-05 & -- \\
$10^{-4}$ & 32 & 2.718e-05 & 2.966 & 3.899e-07 & 5.526 & 2.329e-07 & 4.945 & 2.630e-06 & 5.666 & 5.448e-07 & 5.912 \\
$10^{-4}$ & 64 & 3.528e-06 & 2.946 & 9.646e-09 & 5.337 & 7.696e-09 & 4.919 & 5.889e-08 & 5.481 & 9.624e-09 & 5.823 \\
$10^{-4}$ & 128 & 4.634e-07 & 2.928 & 2.731e-10 & 5.143 & 2.604e-10 & 4.885 & 1.545e-09 & 5.252 & 2.111e-10 & 5.510 \\
\midrule
$10^{-3}$ & 16 & 2.264e-04 & -- & 1.779e-05 & -- & 8.665e-06 & -- & 1.334e-04 & -- & 3.179e-05 & -- \\
$10^{-3}$ & 32 & 2.927e-05 & 2.951 & 4.137e-07 & 5.427 & 2.907e-07 & 4.897 & 2.772e-06 & 5.589 & 5.350e-07 & 5.893 \\
$10^{-3}$ & 64 & 3.514e-06 & 3.058 & 1.088e-08 & 5.248 & 9.583e-09 & 4.923 & 6.596e-08 & 5.393 & 9.572e-09 & 5.805 \\
$10^{-3}$ & 128 & 4.067e-07 & 3.111 & 3.521e-10 & 4.950 & 3.463e-10 & 4.790 & 1.977e-09 & 5.060 & 1.981e-10 & 5.595 \\
\midrule
$10^{-2}$ & 16 & 1.491e-04 & -- & 1.232e-05 & -- & 1.339e-05 & -- & 1.062e-04 & -- & 2.339e-05 & -- \\
$10^{-2}$ & 32 & 1.803e-05 & 3.048 & 3.743e-07 & 5.041 & 5.194e-07 & 4.689 & 2.905e-06 & 5.193 & 3.884e-07 & 5.912 \\
$10^{-2}$ & 64 & 2.220e-06 & 3.022 & 1.821e-08 & 4.361 & 2.696e-08 & 4.268 & 1.238e-07 & 4.552 & 6.760e-09 & 5.844 \\
$10^{-2}$ & 128 & 2.759e-07 & 3.008 & 1.053e-09 & 4.112 & 1.592e-09 & 4.082 & 6.859e-09 & 4.174 & 1.388e-10 & 5.605 \\
\midrule
$10^{-1}$ & 16 & 5.015e-05 & -- & 1.195e-05 & -- & 2.019e-05 & -- & 4.201e-05 & -- & 8.292e-06 & -- \\
$10^{-1}$ & 32 & 6.268e-06 & 3.000 & 3.493e-07 & 5.096 & 8.675e-07 & 4.540 & 1.489e-06 & 4.818 & 1.376e-07 & 5.913 \\
$10^{-1}$ & 64 & 7.848e-07 & 2.998 & 1.792e-08 & 4.285 & 4.775e-08 & 4.183 & 7.702e-08 & 4.273 & 2.396e-09 & 5.844 \\
$10^{-1}$ & 128 & 9.850e-08 & 2.994 & 1.065e-09 & 4.072 & 2.891e-09 & 4.046 & 4.578e-09 & 4.072 & 4.938e-11 & 5.601 \\
\end{longtable}
}

\subsection{Viscous Burgers equation}
\label{subsec:numerical-viscous-burgers-2d}

We consider the viscous Burgers equation
\begin{equation*}
\partial_t u + \partial_x\!\left(\tfrac12 u^2\right) + \partial_y\!\left(\tfrac12 u^2\right)
= \varepsilon \Delta u + s
\end{equation*}
on $(0,T)\times\mathbb{T}^2$ with $T=0.1$.
Here the source term $s$ is chosen by the method of manufactured solutions so that the exact solution is
\begin{equation*}
u(t,x,y)=2+0.2\sin\bigl(2\pi(x-0.7t)\bigr)\cos\bigl(2\pi(y+0.4t)\bigr).
\end{equation*}
Since $u(t,x,y)\in[1.8,2.2]$ for the manufactured solution, we take $\lambda_{\max}=2.2$.
The initial condition is given by $u(0,x,y)$.

Tables~\ref{tab:burgers-2d-q1-values-eoc} and \ref{tab:burgers-2d-q2-values-eoc} present the values and EoCs of $\|\widehat{u}_h^t-u\|_{L^2(0,T;L^2(\mathbb{T}^d))}$, $\|\widehat{u}^{ts}-u\|_{L^\infty(0,T;L^2(\mathbb{T}^d))}$, $\|r_1\|_{L^1(0,T;L^2(\mathbb{T}^d))}$, $|\widehat{u}^{ts}-u|_{L^2(0,T;H^1(\mathbb{T}^d))}$, and $\CE_{r_2}$ defined in \eqref{eq:r2-H-1-bound} for $q=1$ and $q=2$.
We observe that the time reconstruction of the dG solution converges with order $q+1$.
The space--time reconstruction exhibits superconvergence, with observed order around $2q+1$ for the smaller viscosities and dropping toward $2q$ for the larger viscosities.
In particular, the tables show that $r_1$ is of the same order as $\|\widehat{u}^{ts}-u\|_{L^\infty(0,T;L^2(\mathbb{T}^d))}$.
Moreover, $\CE_{r_2}$ is of the same order as $|\widehat{u}^{ts}-u|_{L^2(0,T;H^1(\mathbb{T}^d))}$ at the smaller viscosities and of higher order for the larger viscosities.
The rows with $\varepsilon=0$ show that the scaling behavior persists in the vanishing-viscosity limit.
This provides numerical evidence that the estimator remains robust as $\varepsilon\to 0$ and can be used in the hyperbolic case.
The columns $|\widehat{u}^{ts}-u|_{L^2(0,T;H^1(\mathbb{T}^d))}$ and $\CE_{r_2}$ are left as \texttt{--}, since they represent the diffusion-dependent part of the estimator and therefore disappear in the inviscid case.

{
\renewcommand{\arraystretch}{1.08}
\setlength{\tabcolsep}{3.1pt}
\setlength{\LTleft}{\fill}
\setlength{\LTright}{\fill}

\begin{longtable}{ccrrrrrrrrrr}
\caption{Viscous Burgers equation for polynomial degree $q=1$: values and EoCs of $\|\widehat{u}_h^t-u\|_{L^2(0,T;L^2(\mathbb{T}^d))}$, $\|\widehat{u}^{ts}-u\|_{L^\infty(0,T;L^2(\mathbb{T}^d))}$, $\|r_1\|_{L^1(0,T;L^2(\mathbb{T}^d))}$, $|\widehat{u}^{ts}-u|_{L^2(0,T;H^1(\mathbb{T}^d))}$, and $\CE_{r_2}$.}
\label{tab:burgers-2d-q1-values-eoc}\\
\toprule
\multirow{2}{*}{$\varepsilon$} & \multirow{2}{*}{$N$}
& \multicolumn{2}{c}{\begin{tabular}[c]{@{}c@{}}$\|\widehat{u}_h^t-u\|$\\$\scriptstyle{}_{L^2(0,T;L^2(\mathbb{T}^d))}$\end{tabular}}
& \multicolumn{2}{c}{\begin{tabular}[c]{@{}c@{}}$\|\widehat{u}^{ts}-u\|$\\$\scriptstyle{}_{L^\infty(0,T;L^2(\mathbb{T}^d))}$\end{tabular}}
& \multicolumn{2}{c}{\begin{tabular}[c]{@{}c@{}}$\|r_1\|$\\$\scriptstyle{}_{L^1(0,T;L^2(\mathbb{T}^d))}$\end{tabular}}
& \multicolumn{2}{c}{\begin{tabular}[c]{@{}c@{}}$|\widehat{u}^{ts}-u|$\\$\scriptstyle{}_{L^2(0,T;H^1(\mathbb{T}^d))}$\end{tabular}}
& \multicolumn{2}{c}{$\CE_{r_2}$} \\
\cmidrule(lr){3-4}\cmidrule(lr){5-6}\cmidrule(lr){7-8}\cmidrule(lr){9-10}\cmidrule(lr){11-12}
& & \multicolumn{1}{c}{Value} & \multicolumn{1}{c}{EoC} & \multicolumn{1}{c}{Value} & \multicolumn{1}{c}{EoC} & \multicolumn{1}{c}{Value} & \multicolumn{1}{c}{EoC} & \multicolumn{1}{c}{Value} & \multicolumn{1}{c}{EoC} & \multicolumn{1}{c}{Value} & \multicolumn{1}{c}{EoC} \\
\midrule
\endfirsthead
\toprule
\multirow{2}{*}{$\varepsilon$} & \multirow{2}{*}{$N$}
& \multicolumn{2}{c}{\begin{tabular}[c]{@{}c@{}}$\|\widehat{u}_h^t-u\|$\\$\scriptstyle{}_{L^2(0,T;L^2(\mathbb{T}^d))}$\end{tabular}}
& \multicolumn{2}{c}{\begin{tabular}[c]{@{}c@{}}$\|\widehat{u}^{ts}-u\|$\\$\scriptstyle{}_{L^\infty(0,T;L^2(\mathbb{T}^d))}$\end{tabular}}
& \multicolumn{2}{c}{\begin{tabular}[c]{@{}c@{}}$\|r_1\|$\\$\scriptstyle{}_{L^1(0,T;L^2(\mathbb{T}^d))}$\end{tabular}}
& \multicolumn{2}{c}{\begin{tabular}[c]{@{}c@{}}$|\widehat{u}^{ts}-u|$\\$\scriptstyle{}_{L^2(0,T;H^1(\mathbb{T}^d))}$\end{tabular}}
& \multicolumn{2}{c}{$\CE_{r_2}$} \\
\cmidrule(lr){3-4}\cmidrule(lr){5-6}\cmidrule(lr){7-8}\cmidrule(lr){9-10}\cmidrule(lr){11-12}
& & \multicolumn{1}{c}{Value} & \multicolumn{1}{c}{EoC} & \multicolumn{1}{c}{Value} & \multicolumn{1}{c}{EoC} & \multicolumn{1}{c}{Value} & \multicolumn{1}{c}{EoC} & \multicolumn{1}{c}{Value} & \multicolumn{1}{c}{EoC} & \multicolumn{1}{c}{Value} & \multicolumn{1}{c}{EoC} \\
\midrule
\endhead
\midrule
\endfoot
\bottomrule
\endlastfoot
$0$ & 16 & 4.076e-04 & -- & 2.212e-04 & -- & 2.378e-04 & -- & -- & -- & -- & -- \\
$0$ & 32 & 1.035e-04 & 1.977 & 2.564e-05 & 3.109 & 2.884e-05 & 3.044 & -- & -- & -- & -- \\
$0$ & 64 & 2.606e-05 & 1.990 & 3.079e-06 & 3.058 & 3.599e-06 & 3.002 & -- & -- & -- & -- \\
$0$ & 128 & 6.537e-06 & 1.995 & 3.775e-07 & 3.028 & 4.515e-07 & 2.995 & -- & -- & -- & -- \\
\midrule
$10^{-4}$ & 16 & 4.052e-04 & -- & 2.203e-04 & -- & 2.368e-04 & -- & 4.467e-04 & -- & 1.484e-04 & -- \\
$10^{-4}$ & 32 & 1.022e-04 & 1.988 & 2.537e-05 & 3.118 & 2.853e-05 & 3.053 & 4.970e-05 & 3.168 & 1.789e-05 & 3.053 \\
$10^{-4}$ & 64 & 2.538e-05 & 2.009 & 3.012e-06 & 3.074 & 3.520e-06 & 3.019 & 5.800e-06 & 3.099 & 2.197e-06 & 3.025 \\
$10^{-4}$ & 128 & 6.208e-06 & 2.031 & 3.617e-07 & 3.058 & 4.325e-07 & 3.025 & 6.906e-07 & 3.070 & 2.688e-07 & 3.031 \\
\midrule
$10^{-3}$ & 16 & 3.853e-04 & -- & 2.132e-04 & -- & 2.296e-04 & -- & 4.353e-04 & -- & 1.431e-04 & -- \\
$10^{-3}$ & 32 & 9.261e-05 & 2.057 & 2.373e-05 & 3.168 & 2.666e-05 & 3.106 & 4.679e-05 & 3.218 & 1.652e-05 & 3.114 \\
$10^{-3}$ & 64 & 2.146e-05 & 2.110 & 2.742e-06 & 3.114 & 3.186e-06 & 3.065 & 5.279e-06 & 3.148 & 1.915e-06 & 3.109 \\
$10^{-3}$ & 128 & 4.854e-06 & 2.144 & 3.446e-07 & 2.992 & 4.045e-07 & 2.977 & 6.440e-07 & 3.035 & 2.203e-07 & 3.119 \\
\midrule
$10^{-2}$ & 16 & 3.026e-04 & -- & 2.136e-04 & -- & 2.361e-04 & -- & 4.363e-04 & -- & 1.217e-04 & -- \\
$10^{-2}$ & 32 & 6.945e-05 & 2.124 & 3.123e-05 & 2.774 & 3.543e-05 & 2.737 & 5.849e-05 & 2.899 & 1.334e-05 & 3.190 \\
$10^{-2}$ & 64 & 1.655e-05 & 2.069 & 6.102e-06 & 2.356 & 7.055e-06 & 2.328 & 1.081e-05 & 2.436 & 1.574e-06 & 3.083 \\
$10^{-2}$ & 128 & 4.059e-06 & 2.028 & 1.403e-06 & 2.121 & 1.634e-06 & 2.111 & 2.432e-06 & 2.151 & 1.927e-07 & 3.029 \\
\midrule
$10^{-1}$ & 16 & 2.921e-04 & -- & 6.767e-04 & -- & 1.109e-03 & -- & 1.317e-03 & -- & 1.115e-04 & -- \\
$10^{-1}$ & 32 & 7.243e-05 & 2.012 & 1.598e-04 & 2.082 & 2.595e-04 & 2.095 & 3.022e-04 & 2.123 & 1.270e-05 & 3.134 \\
$10^{-1}$ & 64 & 1.807e-05 & 2.003 & 3.937e-05 & 2.021 & 6.378e-05 & 2.025 & 7.388e-05 & 2.032 & 1.544e-06 & 3.040 \\
$10^{-1}$ & 128 & 4.515e-06 & 2.001 & 9.806e-06 & 2.005 & 1.588e-05 & 2.006 & 1.836e-05 & 2.008 & 1.916e-07 & 3.011 \\
\end{longtable}
}

{
\renewcommand{\arraystretch}{1.08}
\setlength{\tabcolsep}{3.1pt}
\setlength{\LTleft}{\fill}
\setlength{\LTright}{\fill}

\begin{longtable}{ccrrrrrrrrrr}
\caption{Viscous Burgers equation for polynomial degree $q=2$: values and EoCs of $\|\widehat{u}_h^t-u\|_{L^2(0,T;L^2(\mathbb{T}^d))}$, $\|\widehat{u}^{ts}-u\|_{L^\infty(0,T;L^2(\mathbb{T}^d))}$, $\|r_1\|_{L^1(0,T;L^2(\mathbb{T}^d))}$, $|\widehat{u}^{ts}-u|_{L^2(0,T;H^1(\mathbb{T}^d))}$, and $\CE_{r_2}$.}
\label{tab:burgers-2d-q2-values-eoc}\\
\toprule
\multirow{2}{*}{$\varepsilon$} & \multirow{2}{*}{$N$}
& \multicolumn{2}{c}{\begin{tabular}[c]{@{}c@{}}$\|\widehat{u}_h^t-u\|$\\$\scriptstyle{}_{L^2(0,T;L^2(\mathbb{T}^d))}$\end{tabular}}
& \multicolumn{2}{c}{\begin{tabular}[c]{@{}c@{}}$\|\widehat{u}^{ts}-u\|$\\$\scriptstyle{}_{L^\infty(0,T;L^2(\mathbb{T}^d))}$\end{tabular}}
& \multicolumn{2}{c}{\begin{tabular}[c]{@{}c@{}}$\|r_1\|$\\$\scriptstyle{}_{L^1(0,T;L^2(\mathbb{T}^d))}$\end{tabular}}
& \multicolumn{2}{c}{\begin{tabular}[c]{@{}c@{}}$|\widehat{u}^{ts}-u|$\\$\scriptstyle{}_{L^2(0,T;H^1(\mathbb{T}^d))}$\end{tabular}}
& \multicolumn{2}{c}{$\CE_{r_2}$} \\
\cmidrule(lr){3-4}\cmidrule(lr){5-6}\cmidrule(lr){7-8}\cmidrule(lr){9-10}\cmidrule(lr){11-12}
& & \multicolumn{1}{c}{Value} & \multicolumn{1}{c}{EoC} & \multicolumn{1}{c}{Value} & \multicolumn{1}{c}{EoC} & \multicolumn{1}{c}{Value} & \multicolumn{1}{c}{EoC} & \multicolumn{1}{c}{Value} & \multicolumn{1}{c}{EoC} & \multicolumn{1}{c}{Value} & \multicolumn{1}{c}{EoC} \\
\midrule
\endfirsthead
\toprule
\multirow{2}{*}{$\varepsilon$} & \multirow{2}{*}{$N$}
& \multicolumn{2}{c}{\begin{tabular}[c]{@{}c@{}}$\|\widehat{u}_h^t-u\|$\\$\scriptstyle{}_{L^2(0,T;L^2(\mathbb{T}^d))}$\end{tabular}}
& \multicolumn{2}{c}{\begin{tabular}[c]{@{}c@{}}$\|\widehat{u}^{ts}-u\|$\\$\scriptstyle{}_{L^\infty(0,T;L^2(\mathbb{T}^d))}$\end{tabular}}
& \multicolumn{2}{c}{\begin{tabular}[c]{@{}c@{}}$\|r_1\|$\\$\scriptstyle{}_{L^1(0,T;L^2(\mathbb{T}^d))}$\end{tabular}}
& \multicolumn{2}{c}{\begin{tabular}[c]{@{}c@{}}$|\widehat{u}^{ts}-u|$\\$\scriptstyle{}_{L^2(0,T;H^1(\mathbb{T}^d))}$\end{tabular}}
& \multicolumn{2}{c}{$\CE_{r_2}$} \\
\cmidrule(lr){3-4}\cmidrule(lr){5-6}\cmidrule(lr){7-8}\cmidrule(lr){9-10}\cmidrule(lr){11-12}
& & \multicolumn{1}{c}{Value} & \multicolumn{1}{c}{EoC} & \multicolumn{1}{c}{Value} & \multicolumn{1}{c}{EoC} & \multicolumn{1}{c}{Value} & \multicolumn{1}{c}{EoC} & \multicolumn{1}{c}{Value} & \multicolumn{1}{c}{EoC} & \multicolumn{1}{c}{Value} & \multicolumn{1}{c}{EoC} \\
\midrule
\endhead
\midrule
\endfoot
\bottomrule
\endlastfoot
$0$ & 16 & 1.324e-05 & -- & 2.585e-06 & -- & 4.113e-06 & -- & -- & -- & -- & -- \\
$0$ & 32 & 1.653e-06 & 3.001 & 4.472e-08 & 5.853 & 6.647e-08 & 5.951 & -- & -- & -- & -- \\
$0$ & 64 & 2.066e-07 & 3.001 & 8.197e-10 & 5.770 & 1.092e-09 & 5.927 & -- & -- & -- & -- \\
$0$ & 128 & 2.582e-08 & 3.000 & 1.681e-11 & 5.607 & 1.984e-11 & 5.783 & -- & -- & -- & -- \\
\midrule
$10^{-4}$ & 16 & 1.333e-05 & -- & 2.588e-06 & -- & 4.114e-06 & -- & 7.049e-06 & -- & 2.083e-06 & -- \\
$10^{-4}$ & 32 & 1.679e-06 & 2.989 & 4.492e-08 & 5.848 & 6.655e-08 & 5.950 & 1.191e-07 & 5.887 & 3.450e-08 & 5.916 \\
$10^{-4}$ & 64 & 2.130e-07 & 2.978 & 8.330e-10 & 5.753 & 1.100e-09 & 5.919 & 2.104e-09 & 5.823 & 5.999e-10 & 5.846 \\
$10^{-4}$ & 128 & 2.738e-08 & 2.960 & 1.766e-11 & 5.560 & 2.054e-11 & 5.743 & 4.159e-11 & 5.661 & 1.256e-11 & 5.577 \\
\midrule
$10^{-3}$ & 16 & 1.403e-05 & -- & 2.612e-06 & -- & 4.127e-06 & -- & 7.088e-06 & -- & 2.086e-06 & -- \\
$10^{-3}$ & 32 & 1.851e-06 & 2.923 & 4.634e-08 & 5.817 & 6.726e-08 & 5.939 & 1.215e-07 & 5.866 & 3.492e-08 & 5.901 \\
$10^{-3}$ & 64 & 2.432e-07 & 2.928 & 9.029e-10 & 5.681 & 1.146e-09 & 5.875 & 2.229e-09 & 5.768 & 6.353e-10 & 5.781 \\
$10^{-3}$ & 128 & 3.008e-08 & 3.015 & 2.031e-11 & 5.474 & 2.297e-11 & 5.641 & 4.648e-11 & 5.584 & 1.404e-11 & 5.500 \\
\midrule
$10^{-2}$ & 16 & 1.462e-05 & -- & 2.678e-06 & -- & 4.195e-06 & -- & 7.215e-06 & -- & 2.093e-06 & -- \\
$10^{-2}$ & 32 & 1.714e-06 & 3.092 & 4.964e-08 & 5.753 & 6.968e-08 & 5.912 & 1.276e-07 & 5.822 & 3.487e-08 & 5.907 \\
$10^{-2}$ & 64 & 2.033e-07 & 3.076 & 1.207e-09 & 5.362 & 1.389e-09 & 5.649 & 2.731e-09 & 5.546 & 6.078e-10 & 5.842 \\
$10^{-2}$ & 128 & 2.486e-08 & 3.032 & 4.736e-11 & 4.672 & 5.133e-11 & 4.758 & 9.107e-11 & 4.906 & 1.247e-11 & 5.608 \\
\midrule
$10^{-1}$ & 16 & 1.267e-05 & -- & 3.372e-06 & -- & 6.559e-06 & -- & 8.555e-06 & -- & 2.084e-06 & -- \\
$10^{-1}$ & 32 & 1.576e-06 & 3.007 & 1.041e-07 & 5.018 & 1.805e-07 & 5.183 & 2.290e-07 & 5.223 & 3.459e-08 & 5.913 \\
$10^{-1}$ & 64 & 1.965e-07 & 3.003 & 4.939e-09 & 4.398 & 8.107e-09 & 4.477 & 9.741e-09 & 4.555 & 6.018e-10 & 5.845 \\
$10^{-1}$ & 128 & 2.454e-08 & 3.001 & 2.855e-10 & 4.113 & 4.635e-10 & 4.129 & 5.418e-10 & 4.168 & 1.237e-11 & 5.605 \\
\end{longtable}
}

\begin{remark}
Let us compare the estimator we obtain here for $\varepsilon=0$ to the estimators for multidimensional scalar hyperbolic conservation laws in \cite{KrönerOhlberger,DednerMakridakisOhlberger}.
Those estimators are based on Kruzhkov's doubling-of-variables argument.
They avoid the exponential dependence on the Lipschitz constant of the solution, so that one obtains meaningful estimates in case the exact solution is discontinuous.
On the other hand, for smooth solutions, doubling of variables leads to a suboptimal scaling of the error estimator: for polynomial degree $q$, the estimator scales only with order $\tfrac{q+1}{2}$.
\end{remark}

\subsection[Diffusive p-system]{Diffusive \texorpdfstring{$p$}{p}-system}
\label{subsec:numerical-p-system-2d}

We consider the diffusive $p$-system for the state $\mathbf{u}=(\tau,v_1,v_2)^\top$,
\begin{equation*}
\partial_t \tau - \partial_x v_1 - \partial_y v_2 = s_1,
\end{equation*}
\begin{equation*}
\partial_t v_1 + \partial_x p(\tau) = \varepsilon \Delta v_1 + s_2,
\qquad
\partial_t v_2 + \partial_y p(\tau) = \varepsilon \Delta v_2 + s_3,
\end{equation*}
on $(0,T)\times\mathbb{T}^2$ with $T=0.05$ and $p(\tau)=\tau^{-2}$.

The diffusivity acts only on the velocity components $v_1$ and $v_2$, while the $\tau$-equation contains no diffusive term. The source term $\mathbf{s}=(s_1,s_2,s_3)^\top$ is chosen by the method of manufactured solutions so that the exact solution is
\begin{equation*}
\tau(t,x,y)=1+0.1\sin\bigl(2\pi(x-0.3t)\bigr)\cos\bigl(2\pi(y+0.2t)\bigr),
\end{equation*}
\begin{equation*}
v_1(t,x,y)=0.2\cos\bigl(2\pi(x-0.3t)\bigr)\cos\bigl(2\pi(y+0.2t)\bigr),
\qquad
v_2(t,x,y)=-0.15\sin\bigl(2\pi(x-0.3t)\bigr)\sin\bigl(2\pi(y+0.2t)\bigr).
\end{equation*}
Writing $\mathbf{v}:=(v_1,v_2)^\top$ for the velocity vector, we compute the error in the $L^2(0,T;H^1(\mathbb{T}^d))$ seminorm only for $\mathbf{v}$, consistent with the diffusion operator. 
Since $\tau(t,x,y)\in[0.9,1.1]$ for the manufactured solution, we take
\begin{equation*}
\lambda_{\max}=\sqrt{2}\,(0.9)^{-3/2},
\end{equation*}
The initial condition is given by $\mathbf{u}(0,x,y)$.

Tables~\ref{tab:psystem-2d-q1-values-eoc} and \ref{tab:psystem-2d-q2-values-eoc} present the values and EoCs of $\|\widehat{\mathbf{u}}_h^t-\mathbf{u}\|_{L^2(0,T;L^2(\mathbb{T}^d))}$, $\|\widehat{\mathbf{u}}^{ts}-\mathbf{u}\|_{L^\infty(0,T;L^2(\mathbb{T}^d))}$, $\|r_1\|_{L^1(0,T;L^2(\mathbb{T}^d))}$, $|\widehat{\mathbf{v}}^{ts}-\mathbf{v}|_{L^2(0,T;H^1(\mathbb{T}^d))}$, and $\CE_{r_2}$ defined in \eqref{eq:r2-H-1-bound} for $q=1$ and $q=2$.
We observe that the time reconstruction of the dG solution converges with order $q+1$ for $q=1$ and for $q=2$ at the smaller viscosities.
The space--time reconstruction exhibits superconvergence, with observed order around $2q+1$ for the smaller viscosities and dropping toward $2q$ for the larger viscosities.
For $q=2$, the observed order of the time reconstruction also decreases for the larger viscosities.
In particular, the tables show that $r_1$ is of the same order as $\|\widehat{\mathbf{u}}^{ts}-\mathbf{u}\|_{L^\infty(0,T;L^2(\mathbb{T}^d))}$.
Moreover, $\CE_{r_2}$ is of the same order as $|\widehat{\mathbf{v}}^{ts}-\mathbf{v}|_{L^2(0,T;H^1(\mathbb{T}^d))}$ at the smaller viscosities and of higher order for the larger viscosities.
Again, the rows with $\varepsilon=0$ show that the scaling behavior persists in the vanishing-viscosity limit.
This provides numerical evidence that the estimator remains robust as $\varepsilon\to 0$ and can be used for hyperbolic problems.
The columns $|\widehat{\mathbf{v}}^{ts}-\mathbf{v}|_{L^2(0,T;H^1(\mathbb{T}^d))}$ and $\CE_{r_2}$ are left as \texttt{--}, since they represent the diffusion-dependent part of the estimator and therefore disappear in the inviscid case.

{
\renewcommand{\arraystretch}{1.08}
\setlength{\tabcolsep}{3.1pt}
\setlength{\LTleft}{\fill}
\setlength{\LTright}{\fill}

\begin{longtable}{ccrrrrrrrrrr}
\caption{Diffusive $p$-system for polynomial degree $q=1$: values and EoCs of $\|\widehat{\mathbf{u}}_h^t-\mathbf{u}\|_{L^2(0,T;L^2(\mathbb{T}^d))}$, $\|\widehat{\mathbf{u}}^{ts}-\mathbf{u}\|_{L^\infty(0,T;L^2(\mathbb{T}^d))}$, $\|r_1\|_{L^1(0,T;L^2(\mathbb{T}^d))}$, $|\widehat{\mathbf{v}}^{ts}-\mathbf{v}|_{L^2(0,T;H^1(\mathbb{T}^d))}$, and $\CE_{r_2}$.}
\label{tab:psystem-2d-q1-values-eoc}\\
\toprule
\multirow{2}{*}{$\varepsilon$} & \multirow{2}{*}{$N$}
& \multicolumn{2}{c}{\begin{tabular}[c]{@{}c@{}}$\|\widehat{\mathbf{u}}_h^t-\mathbf{u}\|$\\$\scriptstyle{}_{L^2(0,T;L^2(\mathbb{T}^d))}$\end{tabular}}
& \multicolumn{2}{c}{\begin{tabular}[c]{@{}c@{}}$\|\widehat{\mathbf{u}}^{ts}-\mathbf{u}\|$\\$\scriptstyle{}_{L^\infty(0,T;L^2(\mathbb{T}^d))}$\end{tabular}}
& \multicolumn{2}{c}{\begin{tabular}[c]{@{}c@{}}$\|r_1\|$\\$\scriptstyle{}_{L^1(0,T;L^2(\mathbb{T}^d))}$\end{tabular}}
& \multicolumn{2}{c}{\begin{tabular}[c]{@{}c@{}}$|\widehat{\mathbf{v}}^{ts}-\mathbf{v}|$\\$\scriptstyle{}_{L^2(0,T;H^1(\mathbb{T}^d))}$\end{tabular}}
& \multicolumn{2}{c}{$\CE_{r_2}$} \\
\cmidrule(lr){3-4}\cmidrule(lr){5-6}\cmidrule(lr){7-8}\cmidrule(lr){9-10}\cmidrule(lr){11-12}
& & \multicolumn{1}{c}{Value} & \multicolumn{1}{c}{EoC} & \multicolumn{1}{c}{Value} & \multicolumn{1}{c}{EoC} & \multicolumn{1}{c}{Value} & \multicolumn{1}{c}{EoC} & \multicolumn{1}{c}{Value} & \multicolumn{1}{c}{EoC} & \multicolumn{1}{c}{Value} & \multicolumn{1}{c}{EoC} \\
\midrule
\endfirsthead
\toprule
\multirow{2}{*}{$\varepsilon$} & \multirow{2}{*}{$N$}
& \multicolumn{2}{c}{\begin{tabular}[c]{@{}c@{}}$\|\widehat{\mathbf{u}}_h^t-\mathbf{u}\|$\\$\scriptstyle{}_{L^2(0,T;L^2(\mathbb{T}^d))}$\end{tabular}}
& \multicolumn{2}{c}{\begin{tabular}[c]{@{}c@{}}$\|\widehat{\mathbf{u}}^{ts}-\mathbf{u}\|$\\$\scriptstyle{}_{L^\infty(0,T;L^2(\mathbb{T}^d))}$\end{tabular}}
& \multicolumn{2}{c}{\begin{tabular}[c]{@{}c@{}}$\|r_1\|$\\$\scriptstyle{}_{L^1(0,T;L^2(\mathbb{T}^d))}$\end{tabular}}
& \multicolumn{2}{c}{\begin{tabular}[c]{@{}c@{}}$|\widehat{\mathbf{v}}^{ts}-\mathbf{v}|$\\$\scriptstyle{}_{L^2(0,T;H^1(\mathbb{T}^d))}$\end{tabular}}
& \multicolumn{2}{c}{$\CE_{r_2}$} \\
\cmidrule(lr){3-4}\cmidrule(lr){5-6}\cmidrule(lr){7-8}\cmidrule(lr){9-10}\cmidrule(lr){11-12}
& & \multicolumn{1}{c}{Value} & \multicolumn{1}{c}{EoC} & \multicolumn{1}{c}{Value} & \multicolumn{1}{c}{EoC} & \multicolumn{1}{c}{Value} & \multicolumn{1}{c}{EoC} & \multicolumn{1}{c}{Value} & \multicolumn{1}{c}{EoC} & \multicolumn{1}{c}{Value} & \multicolumn{1}{c}{EoC} \\
\midrule
\endhead
\midrule
\endfoot
\bottomrule
\endlastfoot
$0$ & 16 & 3.159e-04 & -- & 1.307e-04 & -- & 7.296e-05 & -- & -- & -- & -- & -- \\
$0$ & 32 & 8.101e-05 & 1.963 & 1.247e-05 & 3.389 & 8.489e-06 & 3.103 & -- & -- & -- & -- \\
$0$ & 64 & 2.049e-05 & 1.983 & 1.314e-06 & 3.247 & 1.061e-06 & 3.000 & -- & -- & -- & -- \\
$0$ & 128 & 5.153e-06 & 1.992 & 1.490e-07 & 3.140 & 1.341e-07 & 2.984 & -- & -- & -- & -- \\
\midrule
$10^{-4}$ & 16 & 3.150e-04 & -- & 1.309e-04 & -- & 7.269e-05 & -- & 2.078e-04 & -- & 1.084e-04 & -- \\
$10^{-4}$ & 32 & 8.045e-05 & 1.969 & 1.250e-05 & 3.388 & 8.462e-06 & 3.103 & 2.012e-05 & 3.369 & 1.281e-05 & 3.081 \\
$10^{-4}$ & 64 & 2.020e-05 & 1.993 & 1.320e-06 & 3.244 & 1.062e-06 & 2.994 & 2.214e-06 & 3.183 & 1.574e-06 & 3.025 \\
$10^{-4}$ & 128 & 5.015e-06 & 2.010 & 1.508e-07 & 3.130 & 1.357e-07 & 2.968 & 2.613e-07 & 3.083 & 1.940e-07 & 3.020 \\
\midrule
$10^{-3}$ & 16 & 3.078e-04 & -- & 1.331e-04 & -- & 7.103e-05 & -- & 2.115e-04 & -- & 1.062e-04 & -- \\
$10^{-3}$ & 32 & 7.692e-05 & 2.000 & 1.304e-05 & 3.352 & 8.627e-06 & 3.041 & 2.077e-05 & 3.348 & 1.223e-05 & 3.118 \\
$10^{-3}$ & 64 & 1.886e-05 & 2.028 & 1.496e-06 & 3.124 & 1.224e-06 & 2.817 & 2.369e-06 & 3.132 & 1.459e-06 & 3.067 \\
$10^{-3}$ & 128 & 4.602e-06 & 2.035 & 2.100e-07 & 2.833 & 1.981e-07 & 2.627 & 3.100e-07 & 2.934 & 1.761e-07 & 3.051 \\
\midrule
$10^{-2}$ & 16 & 2.874e-04 & -- & 1.722e-04 & -- & 9.955e-05 & -- & 2.596e-04 & -- & 9.939e-05 & -- \\
$10^{-2}$ & 32 & 7.169e-05 & 2.003 & 2.454e-05 & 2.811 & 2.071e-05 & 2.265 & 3.318e-05 & 2.968 & 1.128e-05 & 3.139 \\
$10^{-2}$ & 64 & 1.788e-05 & 2.003 & 4.662e-06 & 2.396 & 4.630e-06 & 2.161 & 5.743e-06 & 2.530 & 1.367e-06 & 3.045 \\
$10^{-2}$ & 128 & 4.466e-06 & 2.002 & 1.033e-06 & 2.175 & 1.078e-06 & 2.102 & 1.217e-06 & 2.239 & 1.694e-07 & 3.012 \\
\midrule
$10^{-1}$ & 16 & 2.944e-04 & -- & 5.981e-04 & -- & 6.978e-04 & -- & 7.715e-04 & -- & 9.804e-05 & -- \\
$10^{-1}$ & 32 & 7.362e-05 & 1.999 & 1.342e-04 & 2.156 & 1.655e-04 & 2.076 & 1.648e-04 & 2.227 & 1.121e-05 & 3.129 \\
$10^{-1}$ & 64 & 1.839e-05 & 2.001 & 3.229e-05 & 2.055 & 4.044e-05 & 2.033 & 3.909e-05 & 2.076 & 1.364e-06 & 3.038 \\
$10^{-1}$ & 128 & 4.594e-06 & 2.001 & 7.959e-06 & 2.020 & 1.000e-05 & 2.015 & 9.596e-06 & 2.026 & 1.693e-07 & 3.010 \\
\end{longtable}
}

{
\renewcommand{\arraystretch}{1.08}
\setlength{\tabcolsep}{3.1pt}
\setlength{\LTleft}{\fill}
\setlength{\LTright}{\fill}

\begin{longtable}{ccrrrrrrrrrr}
\caption{Diffusive $p$-system for polynomial degree $q=2$: values and EoCs of $\|\widehat{\mathbf{u}}_h^t-\mathbf{u}\|_{L^2(0,T;L^2(\mathbb{T}^d))}$, $\|\widehat{\mathbf{u}}^{ts}-\mathbf{u}\|_{L^\infty(0,T;L^2(\mathbb{T}^d))}$, $\|r_1\|_{L^1(0,T;L^2(\mathbb{T}^d))}$, $|\widehat{\mathbf{v}}^{ts}-\mathbf{v}|_{L^2(0,T;H^1(\mathbb{T}^d))}$, and $\CE_{r_2}$.}
\label{tab:psystem-2d-q2-values-eoc}\\
\toprule
\multirow{2}{*}{$\varepsilon$} & \multirow{2}{*}{$N$}
& \multicolumn{2}{c}{\begin{tabular}[c]{@{}c@{}}$\|\widehat{\mathbf{u}}_h^t-\mathbf{u}\|$\\$\scriptstyle{}_{L^2(0,T;L^2(\mathbb{T}^d))}$\end{tabular}}
& \multicolumn{2}{c}{\begin{tabular}[c]{@{}c@{}}$\|\widehat{\mathbf{u}}^{ts}-\mathbf{u}\|$\\$\scriptstyle{}_{L^\infty(0,T;L^2(\mathbb{T}^d))}$\end{tabular}}
& \multicolumn{2}{c}{\begin{tabular}[c]{@{}c@{}}$\|r_1\|$\\$\scriptstyle{}_{L^1(0,T;L^2(\mathbb{T}^d))}$\end{tabular}}
& \multicolumn{2}{c}{\begin{tabular}[c]{@{}c@{}}$|\widehat{\mathbf{v}}^{ts}-\mathbf{v}|$\\$\scriptstyle{}_{L^2(0,T;H^1(\mathbb{T}^d))}$\end{tabular}}
& \multicolumn{2}{c}{$\CE_{r_2}$} \\
\cmidrule(lr){3-4}\cmidrule(lr){5-6}\cmidrule(lr){7-8}\cmidrule(lr){9-10}\cmidrule(lr){11-12}
& & \multicolumn{1}{c}{Value} & \multicolumn{1}{c}{EoC} & \multicolumn{1}{c}{Value} & \multicolumn{1}{c}{EoC} & \multicolumn{1}{c}{Value} & \multicolumn{1}{c}{EoC} & \multicolumn{1}{c}{Value} & \multicolumn{1}{c}{EoC} & \multicolumn{1}{c}{Value} & \multicolumn{1}{c}{EoC} \\
\midrule
\endfirsthead
\toprule
\multirow{2}{*}{$\varepsilon$} & \multirow{2}{*}{$N$}
& \multicolumn{2}{c}{\begin{tabular}[c]{@{}c@{}}$\|\widehat{\mathbf{u}}_h^t-\mathbf{u}\|$\\$\scriptstyle{}_{L^2(0,T;L^2(\mathbb{T}^d))}$\end{tabular}}
& \multicolumn{2}{c}{\begin{tabular}[c]{@{}c@{}}$\|\widehat{\mathbf{u}}^{ts}-\mathbf{u}\|$\\$\scriptstyle{}_{L^\infty(0,T;L^2(\mathbb{T}^d))}$\end{tabular}}
& \multicolumn{2}{c}{\begin{tabular}[c]{@{}c@{}}$\|r_1\|$\\$\scriptstyle{}_{L^1(0,T;L^2(\mathbb{T}^d))}$\end{tabular}}
& \multicolumn{2}{c}{\begin{tabular}[c]{@{}c@{}}$|\widehat{\mathbf{v}}^{ts}-\mathbf{v}|$\\$\scriptstyle{}_{L^2(0,T;H^1(\mathbb{T}^d))}$\end{tabular}}
& \multicolumn{2}{c}{$\CE_{r_2}$} \\
\cmidrule(lr){3-4}\cmidrule(lr){5-6}\cmidrule(lr){7-8}\cmidrule(lr){9-10}\cmidrule(lr){11-12}
& & \multicolumn{1}{c}{Value} & \multicolumn{1}{c}{EoC} & \multicolumn{1}{c}{Value} & \multicolumn{1}{c}{EoC} & \multicolumn{1}{c}{Value} & \multicolumn{1}{c}{EoC} & \multicolumn{1}{c}{Value} & \multicolumn{1}{c}{EoC} & \multicolumn{1}{c}{Value} & \multicolumn{1}{c}{EoC} \\
\midrule
\endhead
\midrule
\endfoot
\bottomrule
\endlastfoot
$0$ & 16 & 1.267e-05 & -- & 3.305e-06 & -- & 1.760e-06 & -- & -- & -- & -- & -- \\
$0$ & 32 & 1.635e-06 & 2.954 & 5.435e-08 & 5.926 & 2.818e-08 & 5.965 & -- & -- & -- & -- \\
$0$ & 64 & 2.074e-07 & 2.979 & 9.043e-10 & 5.909 & 4.510e-10 & 5.965 & -- & -- & -- & -- \\
$0$ & 128 & 2.611e-08 & 2.990 & 1.576e-11 & 5.842 & 7.660e-12 & 5.880 & -- & -- & -- & -- \\
\midrule
$10^{-4}$ & 16 & 1.267e-05 & -- & 3.307e-06 & -- & 1.757e-06 & -- & 6.010e-06 & -- & 1.845e-06 & -- \\
$10^{-4}$ & 32 & 1.634e-06 & 2.956 & 5.446e-08 & 5.924 & 2.807e-08 & 5.968 & 9.776e-08 & 5.942 & 3.072e-08 & 5.908 \\
$10^{-4}$ & 64 & 2.064e-07 & 2.984 & 9.113e-10 & 5.901 & 4.461e-10 & 5.976 & 1.605e-09 & 5.928 & 5.386e-10 & 5.834 \\
$10^{-4}$ & 128 & 2.589e-08 & 2.995 & 1.621e-11 & 5.813 & 7.527e-12 & 5.889 & 2.798e-11 & 5.842 & 1.110e-11 & 5.601 \\
\midrule
$10^{-3}$ & 16 & 1.278e-05 & -- & 3.320e-06 & -- & 1.737e-06 & -- & 6.024e-06 & -- & 1.844e-06 & -- \\
$10^{-3}$ & 32 & 1.669e-06 & 2.937 & 5.538e-08 & 5.905 & 2.725e-08 & 5.994 & 9.881e-08 & 5.930 & 3.059e-08 & 5.914 \\
$10^{-3}$ & 64 & 2.259e-07 & 2.885 & 9.736e-10 & 5.830 & 4.191e-10 & 6.023 & 1.677e-09 & 5.881 & 5.266e-10 & 5.860 \\
$10^{-3}$ & 128 & 3.463e-08 & 2.706 & 2.041e-11 & 5.576 & 9.501e-12 & 5.463 & 3.266e-11 & 5.682 & 1.037e-11 & 5.667 \\
\midrule
$10^{-2}$ & 16 & 1.497e-05 & -- & 3.430e-06 & -- & 1.570e-06 & -- & 6.148e-06 & -- & 1.842e-06 & -- \\
$10^{-2}$ & 32 & 2.760e-06 & 2.440 & 6.356e-08 & 5.754 & 2.422e-08 & 6.019 & 1.077e-07 & 5.835 & 3.041e-08 & 5.920 \\
$10^{-2}$ & 64 & 6.055e-07 & 2.188 & 1.560e-09 & 5.348 & 8.493e-10 & 4.834 & 2.306e-09 & 5.546 & 5.181e-10 & 5.876 \\
$10^{-2}$ & 128 & 1.438e-07 & 2.074 & 6.085e-11 & 4.680 & 5.579e-11 & 3.928 & 7.627e-11 & 4.918 & 1.013e-11 & 5.677 \\
\midrule
$10^{-1}$ & 16 & 2.101e-05 & -- & 4.152e-06 & -- & 2.672e-06 & -- & 7.059e-06 & -- & 1.841e-06 & -- \\
$10^{-1}$ & 32 & 4.790e-06 & 2.133 & 1.135e-07 & 5.194 & 1.022e-07 & 4.709 & 1.677e-07 & 5.395 & 3.050e-08 & 5.916 \\
$10^{-1}$ & 64 & 1.165e-06 & 2.040 & 4.896e-09 & 4.534 & 5.563e-09 & 4.199 & 6.246e-09 & 4.747 & 5.277e-10 & 5.853 \\
$10^{-1}$ & 128 & 2.889e-07 & 2.012 & 2.735e-10 & 4.162 & 3.362e-10 & 4.048 & 3.268e-10 & 4.256 & 1.074e-11 & 5.619 \\
\end{longtable}
}

\section{Conclusion}

We have derived a reliable a~posteriori error estimator for fully discrete RKdG approximations of nonlinear convection--diffusion problems in multiple space dimensions by estimating the error between the entropy weak solution and the space--time reconstruction $\widehat{\mathbf{u}}^{ts}$ defined by SIAC filtering and Hermite interpolation.
While we have no proof of efficiency, our numerical experiments indicate that the error estimator decays with the same order as the error under mesh refinement.
The residual splitting makes the dependence of the estimator on $\varepsilon$ explicit and ensures that the estimator remains robust as $\varepsilon\to 0$, which makes it suitable for convection-dominated regimes and hyperbolic problems. Nevertheless, it should be noted that, for nonlinear problems, there is an implicit dependence of the error bounds on $\varepsilon$ since $\varepsilon$ influences the Lipschitz constants of solutions that enter the error bounds. This is not seen by our numerical experiments since we use manufactured solutions independent of $\varepsilon$.

Finally, if one is not interested in $\| \bu - \widehat{\bu}^{ts}\|$ but in $\|\bu - \bu_h\|$ one can simply use the triangle inequality
\[
 \| \bu -\bu_h\|_{L^\infty(0,T;L^ 2(\mathbb{T}^ d))} \leq  \| \bu - \widehat{\bu}^{ts}\|_{L^\infty(0,T;L^ 2(\mathbb{T}^ d))}+
  \| \bu_h - \widehat{\bu}^{ts}\|_{L^\infty(0,T;L^ 2(\mathbb{T}^ d))}
\]
where the second term on the right hand side is readily computable, and we have provided a computable bound for the first. Our numerical experiments indicate that $\| \bu_h - \widehat{\bu}^{ts}\|_{L^\infty(0,T;L^ 2(\mathbb{T}^ d))}$
coincides with the error $\| \bu -\bu_h\|_{L^\infty(0,T;L^ 2(\mathbb{T}^ d))}$ up to higher order terms, i.e., $\| \bu_h - \widehat{\bu}^{ts}\|_{L^\infty(0,T;L^ 2(\mathbb{T}^ d))}$ constitutes an asymptotically exact error indicator.

\subsection*{Acknowledgements}{Financial support of all authors by the German Research Foundation (DFG), within the project No. 525877563 of the Priority Programme - SPP 2410 Hyperbolic Balance Laws in Fluid Mechanics: Complexity, Scales, Randomness (CoScaRa), is acknowledged.}

\appendix

\section{Butcher Tableaux for the Kennedy--Carpenter IMEX Runge--Kutta Methods}
\label{app:imex-rk-tableaux}

In this appendix, we collect the Butcher tableaux for the Kennedy--Carpenter IMEX Runge--Kutta methods $\mathrm{ARK3(2)4L[2]SA}$ and $\mathrm{ARK5(4)8L[2]SA}$~\cite{2003KennedyCarpenter}. In the numerical experiments, the former is used with the spatial dG discretization of polynomial degree $q=1$, and the latter with polynomial degree $q=2$. Tables~\ref{eq:imexark3-tableau-explicit}, \ref{eq:imexark3-tableau-implicit}, \ref{eq:ark548l2sa-tableau-explicit}, and \ref{eq:ark548l2sa-tableau-implicit} present the explicit and implicit tableaux for these methods.

\begin{table}[htbp]
\centering
\caption{Explicit Butcher tableau for the IMEX RK method $\mathrm{ARK3(2)4L[2]SA}$}
\label{eq:imexark3-tableau-explicit}
\butchertableau{
0 & 0 & 0 & 0 & 0 \\
0.87173304302 & 0.87173304302 & 0 & 0 & 0 \\
0.6 & 0.52758901198 & 0.07241098802 & 0 & 0 \\
1 & 0.39909600768 & -0.43755765461 & 1.03846164694 & 0 \\
\hline
 & 0.18764102435 & -0.59529715734 & 0.97178992772 & 0.43586652151
}
\end{table}

\begin{table}[htbp]
\centering
\caption{Implicit Butcher tableau for the IMEX RK method $\mathrm{ARK3(2)4L[2]SA}$}
\label{eq:imexark3-tableau-implicit}
\butchertableau{
0 & 0 & 0 & 0 & 0 \\
0.87173304302 & 0.43586652150 & 0.43586652150 & 0 & 0 \\
0.6 & 0.25764824607 & -0.09351476757 & 0.43586652150 & 0 \\
1 & 0.18764102435 & -0.59529715734 & 0.97178992772 & 0.43586652150 \\
\hline
 & 0.18764102435 & -0.59529715734 & 0.97178992772 & 0.43586652150
}
\end{table}

\begin{table}[htbp]
\centering
\caption{Explicit Butcher tableau for the IMEX RK method $\mathrm{ARK5(4)8L[2]SA}$}
\label{eq:ark548l2sa-tableau-explicit}
\butchertableau{
0 & 0 & 0 & 0 & 0 & 0 & 0 & 0 & 0 \\
0.41 & 0.41 & 0 & 0 & 0 & 0 & 0 & 0 & 0 \\
0.25992958 & 0.17753521 & 0.08239438 & 0 & 0 & 0 & 0 & 0 & 0 \\
0.19815049 & 0.12262308 & 0 & 0.07552741 & 0 & 0 & 0 & 0 & 0 \\
0.92 & 2.29017765 & 0 & 11.24492577 & -12.61510341 & 0 & 0 & 0 & 0 \\
0.24 & 0.40294452 & 0 & 1.35401238 & -1.4857009 & -0.031256 & 0 & 0 & 0 \\
0.6 & 1.46413844 & 0 & 7.23046868 & -7.84460712 & -0.125 & -0.125 & 0 & 0 \\
1 & -1.674808 & 0 & -6.38943865 & 14.69220068 & 0.09466623 & -7.21115733 & 1.48853707 & 0 \\
\hline
 & -0.09554859 & 0 & 0 & 2.3386928 & -0.14043176 & -2.07058771 & 0.76287525 & 0.205
}
\end{table}

\begin{table}[htbp]
\centering
\caption{Implicit Butcher tableau for the IMEX RK method $\mathrm{ARK5(4)8L[2]SA}$}
\label{eq:ark548l2sa-tableau-implicit}
\butchertableau{
0 & 0 & 0 & 0 & 0 & 0 & 0 & 0 & 0 \\
0.41 & 0.205 & 0.205 & 0 & 0 & 0 & 0 & 0 & 0 \\
0.25992958 & 0.1025 & -0.04757042 & 0.205 & 0 & 0 & 0 & 0 & 0 \\
0.19815049 & 0.07389944 & 0 & -0.08074895 & 0.205 & 0 & 0 & 0 & 0 \\
0.92 & 0.29921812 & 0 & 2.46382067 & -2.04803878 & 0.205 & 0 & 0 & 0 \\
0.24 & 0.14689238 & 0 & 0.11740333 & -0.22170197 & -0.00759375 & 0.205 & 0 & 0 \\
0.6 & 0.1784573 & 0 & 1.01974675 & -0.22154535 & -0.03612492 & -0.54553377 & 0.205 & 0 \\
1 & -0.09554859 & 0 & 0 & 2.3386928 & -0.14043176 & -2.07058771 & 0.76287525 & 0.205 \\
\hline
 & -0.09554859 & 0 & 0 & 2.3386928 & -0.14043176 & -2.07058771 & 0.76287525 & 0.205
}
\end{table}

\bibliographystyle{plainnat}
\bibliography{references}

@article{DednerMakridakisOhlberger,
 author = {Dedner, Andreas and Makridakis, Charalambos and Ohlberger, Mario},
 title = {Error control for a class of {Runge}-{Kutta} discontinuous {Galerkin} methods for nonlinear conservation laws},
 journal = {SIAM J. Numer. Anal.},
 volume = {45},
 number = {2},
 pages = {514--538},
 year = {2007},
 doi = {10.1137/050624248},
}

@article{KrönerOhlberger,
 author = {Kr{\"o}ner, Dietmar and Ohlberger, Mario},
 title = {A posteriori error estimates for upwind finite volume schemes for nonlinear conservation laws in multi dimensions},
 journal = {Math. Comput.},
 volume = {69},
 number = {229},
 pages = {25--39},
 year = {2000},
 doi = {10.1090/S0025-5718-99-01158-8},
}

@Article{2003Cockburn,
  author    = {Cockburn, Bernardo and Luskin, Mitchell and Shu, Chi-Wang and S{\"u}li, Endre},
  journal   = {Mathematics of Computation},
  title     = {Enhanced accuracy by post-processing for finite element methods for hyperbolic equations},
  year      = {2003},
  number    = {242},
  pages     = {577--606},
  volume    = {72},
  doi       = {10.1090/S0025-5718-02-01464-3},
}

@Book{2013Verfurth,
  author    = {Verf{\"u}rth, R{\"u}diger},
  title     = {A Posteriori Error Estimation Techniques for Finite Element Methods},
  publisher = {Oxford University Press},
  year      = {2013},
  doi       = {10.1093/acprof:oso/9780199679423.001.0001},
}

@Article{2008Sangalli,
  author    = {Sangalli, Giancarlo},
  journal   = {Mathematics of Computation},
  title     = {Robust a-posteriori estimator for advection-diffusion-reaction problems},
  year      = {2008},
  number    = {261},
  pages     = {41--70},
  volume    = {77},
  doi       = {10.1090/S0025-5718-07-02018-2},
}

@Article{2014Cangiani,
  author    = {Cangiani, Andrea and Georgoulis, Emmanuil H. and Metcalfe, Stephen},
  journal   = {IMA Journal of Numerical Analysis},
  title     = {Adaptive discontinuous {G}alerkin methods for nonstationary convection-diffusion problems},
  year      = {2014},
  number    = {4},
  pages     = {1578--1597},
  volume    = {34},
  doi       = {10.1093/imanum/drt052},
}

@Article{2009Ryan,
  author    = {Ryan, Jennifer K. and Cockburn, Bernardo},
  journal   = {Journal of Computational Physics},
  title     = {Local derivative post-processing for the discontinuous {G}alerkin method},
  year      = {2009},
  number    = {23},
  pages     = {8642--8664},
  volume    = {228},
  doi       = {10.1016/j.jcp.2009.08.017},
}

@Article{2025Dedner,
  author        = {Dedner, Andreas and Giesselmann, Jan and Kwon, Kiwoong and Pryer, Tristan},
  journal       = {arXiv preprint arXiv:2510.09449},
  title         = {A posteriori analysis for nonlinear convection-diffusion systems},
  year          = {2025},
  archiveprefix = {arXiv},
  doi           = {10.48550/ARXIV.2510.09449},
  eprint        = {2510.09449},
  primaryclass  = {math.NA},
}

@Article{1979Dafermos,
  author     = {Dafermos, Constantine M.},
  journal    = {Archive for Rational Mechanics and Analysis},
  title      = {The second law of thermodynamics and stability},
  year       = {1979},
  number     = {2},
  pages      = {167--179},
  volume     = {70},
  doi        = {10.1007/BF00250353},
}

@Article{1979DiPerna,
  author     = {DiPerna, Ronald J.},
  journal    = {Indiana University Mathematics Journal},
  title      = {Uniqueness of solutions to hyperbolic conservation laws},
  year       = {1979},
  number     = {1},
  pages      = {137--188},
  volume     = {28},
  doi        = {10.1512/iumj.1979.28.28011},
}

@Article{2012Ji,
  author    = {Ji, Liangyue and Xu, Yan and Ryan, Jennifer K.},
  journal   = {Mathematics of Computation},
  title     = {Accuracy-enhancement of discontinuous {G}alerkin solutions for convection-diffusion equations in multiple-dimensions},
  year      = {2012},
  number    = {280},
  pages     = {1929--1950},
  volume    = {81},
  doi       = {10.1090/s0025-5718-2012-02586-5},
}

@Article{2012King,
  author    = {King, James and Mirzaee, Hanieh and Ryan, Jennifer K. and Kirby, Robert M.},
  journal   = {Journal of Scientific Computing},
  title     = {Smoothness-Increasing Accuracy-Conserving ({SIAC}) Filtering for Discontinuous {G}alerkin Solutions: Improved Errors Versus Higher-Order Accuracy},
  year      = {2012},
  number    = {1},
  pages     = {129--149},
  volume    = {53},
  doi       = {10.1007/s10915-012-9593-8},
}

@Book{2012Pietro,
  author    = {Di Pietro, Daniele Antonio and Ern, Alexandre},
  publisher = {Springer},
  title     = {Mathematical Aspects of Discontinuous {G}alerkin Methods},
  year      = {2012},
  doi       = {10.1007/978-3-642-22980-0},
}

@Article{2006HartmannHouston,
  author    = {Hartmann, Ralf and Houston, Paul},
  journal   = {International Journal of Numerical Analysis and Modeling},
  title     = {Symmetric Interior Penalty {DG} Methods for the Compressible {N}avier--{S}tokes Equations {I}: Method Formulation},
  year      = {2006},
  number    = {1},
  pages     = {1--20},
  volume    = {3},
}

@Article{2017Giesselmann,
  author    = {Giesselmann, Jan and Pryer, Tristan},
  journal   = {Mathematical Models and Methods in Applied Sciences},
  title     = {A posteriori analysis for dynamic model adaptation in convection-dominated problems},
  year      = {2017},
  number    = {13},
  pages     = {2381--2423},
  volume    = {27},
  doi       = {10.1142/s0218202517500476},
}

@InBook{2015Ryan,
  author    = {Ryan, Jennifer K.},
  pages     = {87--102},
  publisher = {Springer},
  title     = {Exploiting Superconvergence Through Smoothness-Increasing Accuracy-Conserving (SIAC) Filtering},
  year      = {2015},
  booktitle = {Spectral and High Order Methods for Partial Differential Equations ICOSAHOM 2014},
  doi       = {10.1007/978-3-319-19800-2_6},
}

@Article{2016Meng,
  author    = {Meng, Xiong and Ryan, Jennifer K.},
  journal   = {Numerische Mathematik},
  title     = {Discontinuous {G}alerkin methods for nonlinear scalar hyperbolic conservation laws: divided difference estimates and accuracy enhancement},
  year      = {2016},
  number    = {1},
  pages     = {27--73},
  volume    = {136},
  doi       = {10.1007/s00211-016-0833-y},
}

@Article{2017Meng,
  author    = {Meng, Xiong and Ryan, Jennifer K.},
  journal   = {IMA Journal of Numerical Analysis},
  title     = {Divided difference estimates and accuracy enhancement of discontinuous {G}alerkin methods for nonlinear symmetric systems of hyperbolic conservation laws},
  year      = {2017},
  number    = {1},
  pages     = {125--155},
  volume    = {38},
  doi       = {10.1093/imanum/drw072},
}

@Article{2013Dolejsi,
  author    = {Dolej{\v{s}}{\'i}, V{\'i}t and Ern, Alexandre and Vohral{\'i}k, Martin},
  journal   = {SIAM Journal on Numerical Analysis},
  title     = {A framework for robust a posteriori error control in unsteady nonlinear advection-diffusion problems},
  year      = {2013},
  number    = {2},
  pages     = {773--793},
  volume    = {51},
  doi       = {10.1137/110859282},
}

@Article{2014Ryan,
  author    = {Ryan, Jennifer K. and Li, Xiaozhou and Kirby, Robert M. and Vuik, Kees},
  journal   = {Journal of Scientific Computing},
  title     = {One-sided position-dependent smoothness-increasing accuracy-conserving ({SIAC}) filtering over uniform and non-uniform meshes},
  year      = {2014},
  number    = {3},
  pages     = {773--817},
  volume    = {64},
  doi       = {10.1007/s10915-014-9946-6},
}

@Article{2021Dedner,
  author    = {Dedner, Andreas and Giesselmann, Jan and Pryer, Tristan and Ryan, Jennifer K.},
  journal   = {Journal of Scientific Computing},
  title     = {Residual estimates for post-processors in elliptic problems},
  year      = {2021},
  number    = {2},
  volume    = {88},
  doi       = {10.1007/s10915-021-01502-2},
}

@Article{2016Dedner,
  author     = {Dedner, Andreas and Giesselmann, Jan},
  journal    = {SIAM Journal on Numerical Analysis},
  title      = {A posteriori analysis of fully discrete method of lines discontinuous {G}alerkin schemes for systems of conservation laws},
  year       = {2016},
  number     = {6},
  pages      = {3523--3549},
  volume     = {54},
  doi        = {10.1137/15m1046265},
}

@Article{2015Giesselmann,
  author    = {Giesselmann, Jan and Makridakis, Charalambos and Pryer, Tristan},
  journal   = {SIAM Journal on Numerical Analysis},
  title     = {A posteriori analysis of discontinuous {G}alerkin schemes for systems of hyperbolic conservation laws},
  year      = {2015},
  number    = {3},
  pages     = {1280--1303},
  volume    = {53},
  doi       = {10.1137/140970999},
}

@Article{2003KennedyCarpenter,
  author    = {Kennedy, Christopher A. and Carpenter, Mark H.},
  journal   = {Applied Numerical Mathematics},
  title     = {Additive Runge--Kutta Schemes for Convection--Diffusion--Reaction Equations},
  year      = {2003},
  number    = {1-2},
  pages     = {139--181},
  volume    = {44},
  doi       = {10.1016/S0168-9274(02)00138-1},
}

\end{document}